\documentclass[sigconf]{acmart}
\usepackage{microtype}%if unwanted, comment out or use option "draft"
\usepackage{mathtools}% http://ctan.org/pkg/mathtools
\usepackage{hyperref}% http://ctan.org/pkg/hyperref
\usepackage{latexsym}
\usepackage{amssymb}
\usepackage{amsmath}
\usepackage{amsfonts}
\usepackage{array}
\usepackage{graphicx}
\usepackage[utf8]{inputenc}
\usepackage{tikz}
\usetikzlibrary{positioning}
\usetikzlibrary{patterns}
\usetikzlibrary{calc}

\DeclarePairedDelimiter{\diagfences}{(}{)}
\DeclarePairedDelimiter{\set}{\{}{\}}
\newcommand{\diag}{\operatorname{diag}\diagfences}

\newcommand{\Q}{\mathbb{Q}}
\newcommand{\R}{\mathbb{R}}

\newcommand{\N}{\mathbb{N}}
\newcommand{\Z}{\mathbb{Z}}
\newcommand{\Ext}{\operatorname{Ext}}

\newcommand{\Aff}{\operatorname{Aff}}

\newcommand{\Span}{\operatorname{Span}}

\DeclarePairedDelimiter{\Angle}{\langle}{\rangle}

\newcommand{\calP}{\mathcal{P}}

\newcommand{\boldx}{\boldsymbol{x}}
\newcommand{\boldy}{\boldsymbol{y}}
\newcommand{\boldz}{\boldsymbol{z}}
\newcommand{\bolds}{\boldsymbol{s}}
\newcommand{\boldt}{\boldsymbol{t}}
\newcommand{\boldu}{\boldsymbol{u}}
\newcommand{\bolde}{\boldsymbol{e}}
\newcommand{\boldv}{\boldsymbol{v}}
\newcommand{\boldw}{\boldsymbol{w}}
\newcommand{\boldtau}{\boldsymbol{\tau}}
\newcommand{\boldzero}{\boldsymbol{0}}

\newcommand{\MAX}{\mathrm{max}}
\newcommand{\MIN}{\mathrm{min}}
\newcommand{\CONE}{\mathrm{Cone}}
\newcommand{\DOM}{\mathrm{Dom}}

\definecolor{darkgreen}{rgb}{0.1,0.7,0.1}

\theoremstyle{acmplain}
\newtheorem{theorem}{Theorem}[section]

\newtheorem{proposition}[theorem]{Proposition}
\newtheorem{lemma}[theorem]{Lemma}

\theoremstyle{acmdefinition}

\newtheorem{definition}[theorem]{Definition}
\newtheorem{claim}[theorem]{Claim}

\setcopyright{rightsretained}
\acmDOI{10.475/123_4}

\acmISBN{123-4567-24-567/08/06}

\acmConference[WOODSTOCK'97]{ACM Woodstock conference}{July 1997}{El
  Paso, Texas USA}
\acmYear{1997}
\copyrightyear{2016}

\acmPrice{15.00}

\acmSubmissionID{123-A12-B3}

\title{On the Decidability of Reachability in Linear Time-Invariant Systems}
\author{Nathana\"{e}l Fijalkow}
\affiliation{\institution{CNRS, LaBRI, Bordeaux, France}}
\affiliation{\institution{Alan Turing Institute of Data Science, London, United Kingdom}}
\email{nathanael.fijalkow@labri.fr}

\author{Jo\"{e}l Ouaknine}
\affiliation{Max Planck Institute for Software Systems, Saarland
  Informatics Campus, Germany}
\email{joel@mpi-sws.org}

\author{Amaury Pouly}
\affiliation{CNRS, IRIF, Université Paris Diderot, France}
\email{amaury.pouly@irif.fr}

\author{Jo\~{a}o Sousa-Pinto}
\affiliation{Department of Computer Science, University of Oxford, UK}
\email{jspinto@cs.ox.ac.uk}

\author{James Worrell}
\affiliation{Department of Computer Science, University of Oxford, UK}
\email{jbw@cs.ox.ac.uk}

\renewcommand{\shortauthors}{N. Fijalkow, J. Ouaknine, A. Pouly, J. Sousa-Pinto, and J. Worrell}



\begin{document}

\begin{abstract}
We consider the decidability of state-to-state reachability in linear
time-invariant control systems over discrete time.  We analyse this
problem with respect to the allowable control sets, which in general
are assumed to be defined by boolean combinations of linear inequalities.
Decidability of the version of the reachability problem in which
control sets are affine subspaces of $\R^n$ is a fundamental result in
control theory.  Our first result is that reachability is undecidable
if the set of controls is a finite union of affine subspaces.  We also
consider versions of the reachability problem in which (i)~the set of
controls consists of a single affine subspace together with the origin
and (ii)~the set of controls is a convex polytope.  In these two cases
we respectively show that the reachability problem is as hard as
Skolem's Problem and the Positivity Problem for linear recurrence
sequences (whose decidability has been open for several decades).  Our
main contribution is to show decidability of a version of the
reachability problem in which control sets are convex polytopes, under
certain spectral assumptions on the transition matrix.

\end{abstract}

\maketitle

\begin{acks}
  This project has received funding from the Alan Turing Institute
  under EPSRC grant EP/N510129/1.  Jo\"el Ouaknine was supported by
  ERC grant AVS-ISS (648701), and by the Deutsche
  Forschungsgemeinschaft (DFG, German Research Foundation) ---
  Projektnummer 389792660 --- TRR 248. Jo\"el Ouaknine is also
  affiliated with the Department of Computer Science, University of
  Oxford, UK.  Nathana{\"e}l Fijalkow and Amaury Pouly were supported
  by the CODYS project ANR-18-CE40-0007.  James Worrell was supported
  by EPSRC Fellowship EP/N008197/1.  Amaury Pouly did part of this
  work at the Max Planck Institute for Software Systems, Saarland
  Informatics Campus, Germany.
\end{acks}

\section{Introduction}
This paper is concerned with \emph{linear time-invariant (LTI)
  systems}.  LTI systems are one of the most basic and fundamental
models in control theory and have applications in circuit design,
signal processing, and image processing, among many other areas.  LTI
systems have both discrete-time and continuous-time variants; here we
are concerned solely with the discrete-time version.

A (discrete-time) LTI system in dimension $d$ is specified by a
transition matrix $A \in \mathbb{Q}^{d\times d}$ and a set of controls
$U\subseteq \mathbb{R}^d$.  The evolution of the system is described
by the recurrence $\boldsymbol{x}_{t+1} = A\boldsymbol{x}_t
+ \boldsymbol{u}_t$, where $\boldsymbol{u}_t \in U$ for all
$t\in \mathbb{N}$.  Here we think of the vectors $\boldsymbol{u}_t$ as
inputs that are applied to the system.

Given such an LTI system, we say that state $\bolds \in \R^d$ can
\emph{reach} state $\boldt \in \R^d$ if there exists $T\geq 0$ and a
sequence of controls
$\boldsymbol{u}_0,\ldots,\boldsymbol{u}_{T-1}\in U$ such that the
unique solution to the recurrence
$\boldsymbol{x}_{t+1}=A\boldsymbol{x}_t+\boldsymbol{u}_t$ with initial
condition $\boldsymbol{x}_0 = \bolds$ satisfies
$\boldsymbol{x}_T=\boldsymbol{t}$.  The problem of computing the set
of all states reachable from a given initial state has been an active
topic of research for several decades.  Here the emphasis is typically
on efficient and scalable methods to over- and under-approximate the
reachable
set~\cite{CattaruzzaASK15,GirardGM06,GirardG08,Kaynama2010OverapproximatingTR,SummersWS92}.
By contrast, relatively little attention has been paid to the
\emph{decidability} of the reachable set---the focus of the present
paper.  Specifically we consider the \emph{LTI Reachability Problem}:
given an LTI system, source state $\bolds$, and target state $\boldt$,
decide whether $\bolds$ can reach $\boldt$.  
  The main axis along which we delineate variants of the
LTI Reachability Problem concerns the class of allowable control sets
(e.g., affine subspaces, convex polytopes, etc.).

Other reachability problems on LTI systems include so-called
\emph{null reachability} (can one reach all states from the origin?)
and \emph{null controllability} (can one reach the origin from all
states?)~\cite{BlondelT99a}.  However these ``universal'' reachability
problems have a very different character to the point-to-point version
that we study.  In particular, both null reachability and null
controllable are decidable in polynomial time using linear algebra.

One of the first people to address the LTI Reachability Problem was
Harrison~\cite{Harrison} who posed the question of whether the problem
is decidable when the allowable control sets are vector subspaces of
$\R^d$. (Harrison~\cite{Harrison} called this the
  \emph{accessibility problem for linear sequential machines}.)
Harrison's question was resolved in a seminal paper of Lipton and
Kannan~\cite{KL86}, who gave a polynomial-time procedure for the LTI
Reachability Problem in the case of linear control sets.  The case in
which the allowable control sets are affine subspaces of $\R^d$ (i.e.,
translates of linear subspaces) can easily be reduced to the linear
case by a standard homogenisation trick.

The starting point of the present paper is to give a number of
hardness results for relatively mild generalisations of
Harrison's problem.  Specifically we show that:
\begin{enumerate}
\item If the allowable control sets are 
  finite unions of affine subspaces of $\R^d$, then the LTI
  Reachability Problem is undecidable.
\item If the allowable control sets are of the form $V\cup\set{\boldsymbol{0}}$,
with $V$ an affine subspace of $\R^d$, then the
LTI Reachability Problem is as hard as Skolem's Problem for linear
  recurrence sequences.
\item If the allowable control sets are convex polytopes then
the LTI Reachability Problem is as hard as the Positivity Problem for
  linear recurrence sequences.
\end{enumerate}
Skolem's Problem asks whether a given integer linear recurrence
sequence has a zero term, while the Positivity Problem asks whether
all terms of a given integer linear recurrence sequence are positive.
The decidability of both problems has been open since the
1970s~\cite{SalomaaS78,RS94,Tao08,TUCS05}.  To date, decidability of
Skolem's Problem is known only for recurrences of order at most
4~\cite{Ver85,MST84} and decidability of the Positivity Problem is
known only for recurrences of order at most 5~\cite{PP}.  Thus the
results in this paper suggest that deciding the LTI Reachability
Problem for any class of control sets more general than affine
subspaces will prove a very challenging problem.  Note however that
the problem is straightforwardly semi-decidable, as
reachability in $n$ steps for each fixed $n\in\mathbb{N}$ is easily
reduced to solving a linear program.

\textbf{Main Result.}
Our main result is a decision procedure for a version of the LTI
Reachability Problem in which the initial state is the origin and the
target is a convex polytope (generalising the case of reaching a
single state).  We assume that the control set is a convex polytopic
neighbourhood of the origin.  Intuitively, the condition that
$\boldsymbol{0}$ lie in the interior of the control set ensures that
we can control in every direction.  As one might expect from the above
discussion of hardness, our decision procedure requires fairly strong
hypotheses on the transition matrix $A$ in order to work.
Specifically we assume that (i)~$A$ has spectral radius $\rho(A)<1$
and (ii)~some positive power of $A$ has exclusively real spectrum
(generalising the requirement that $A$ have real spectrum).
Condition~(i) is equivalent to the requirement that the system without
input be asymptotically stable (also called Schur stable).
Condition~(ii) has appeared in closely related contexts, such as
o-minimal hybrid systems (see~\cite[Theorem 6.2]{LafferrierePS00}
and~\cite[Theorem 4.6]{ShakerniaSP00}) and self-affine fractals (see
below).  

As we will show, without loss of generality we can restrict attention
to LTI systems in which the set of vectors reachable from the origin
is full dimensional.  In this case, Condition~(i) and the assumptions
on the control set entail that the set of reachable vectors is a
bounded convex open subset of $\mathbb{R}^n$.  The essential challenge
in deciding reachability is to handle the case in which the target
point lies on the boundary of the reachable set.  Condition (ii) plays
two roles in this respect.  First we use it to show that any
unreachable point is separated from the reachable set by a hyperplane
whose normal vector has algebraic-number coefficients (having
previously observed that hyperplanes with rational normal vectors do
not suffice: see Figure~\ref{fig:convex-lower-dim-face}).  Thus
Condition~(ii) ensures that we have an enumerable set of ``witnesses''
of non-reachability.  Moreover we use this same condition to show that
we can effectively verify such witnesses, i.e.  decide whether some
hyperplane with a given normal vector indeed separates the reachable
set from the target point.

\textbf{Related Work.}  It well understood that most control problems
are undecidable for mild generalisations of linear
systems~\cite{BlondelT99,BlondelT00}.  For example, point-to-point
reachability is undecidable for \emph{piecewise linear
  systems}~\cite{AsarinMP95,BlondelT99a,KoiranCG94} and for
\emph{saturated linear systems}~\cite{SiegelmannS95}.  However to the
best of our knowledge no previous work has attempted to systematically
map the border of decidability for point-to-point reachability
\emph{within} the class of LTI systems.  Indeed it is sometimes
considered that point-to-point reachability is efficiently decidable
for LTI systems (see, e.g., the discussion in~\cite[Section
4.1]{BlondelT00}).  The results of this paper illustrate that the
latter view crucially depends on the assumption that the set of
controls $U$ form a linear (or affine) subspace.  But such an
assumption does not allow to express many natural requirements, e.g.,
that $U$ be bounded.

A range of different control problems for discrete- and
continuous-time LTI systems under constraints on the set of controls
have been studied in the
literature~\cite{Cook80,Sontag84,SummersWS92,HuQiu98,HuMQ02,HuLQ02,TilS86,Grantham1975,SchmitendorfB80,HeemelsC08,
GirardG08,Jamak00,Zhao17}.  Although it is very common to consider
systems with saturated inputs, we do not consider the problem of
controller design and thus saturated inputs reduce to having inputs in
the unit hypercube in our case.  LTI systems with convex input
constraints have been considered in the past (see the references
above) but we are not aware of any complete characterisation of the
reachable set in this case, except in the case of conical constraints.

There is a clear relationship between the LTI Reachability Problem with
bounded convex control sets and self-affine fractals.  For LTI systems
with spectral radius $\rho(A)<1$ and with input set $U$ a convex
polytope, the closure of the set of states reachable from
$\boldsymbol{0}$ is the convex hull of the self-affine
fractal $\mathcal{F}$ arising as the unique solution of the set equation
$\mathcal{F}=A\mathcal{F}+\Ext(U)$, where $\Ext(U)$ denotes the set of
extreme points of $U$. We are aware of several
results~\cite{KiratK2015,Vass2015} on the computability of the convex
hulls of such fractals (and more general types of fractals). However
those results only apply to the case when the convex hull is a
polytope and usually only in dimension $2$.  The requirements on the
spectrum of $A$ in our positive decidability result are related to the
so-called fractal of unity of~\cite{Vass2015}.

\section{Undecidability and Hardness}

In this section we give evidence for the hardness of the LTI
reachability problem.  We show undecidability if the set of controls
is a finite union of affine subspaces, we give a reduction from the
Positivity Problem in case the set of controls is a bounded convex
polytope, and we give a reduction from Skolem's Problem in case the
set of controls is a union of two affine subspaces.

\subsection{Undecidability}
The goal of this subsection is to prove the following result.
\begin{theorem}\label{thm:undecidability}
  The reachability problem for LTI systems whose sets of controls are
  finite unions of affine subspaces is undecidable.
\end{theorem}

We prove Theorem~\ref{thm:undecidability} by reduction from the
\emph{vector reachability problem for invertible matrices}: given
invertible matrices $A_1, \ldots, A_k \in \Q^{d \times d}$ and vectors
$\boldx, \boldy \in \Q^d$, do there exist integers $n_1,
\ldots, n_k$ such that $\prod\limits_{i=1}^k A_i^{n_{i}} \boldx =
\boldy$?  The undecidability of this problem is folklore
(but see~\cite{FOPPW18} for a proof).  The key
idea underlying the reduction of this problem to the reachability
problem for LTI systems is to form an LTI whose transition matrix $A$
incorporates $A_1,\ldots,A_k$, and to provide a set of controls that
can be used to simulate the successive application of powers of $A_1$,
$A_2$, \emph{etc}, by repeated application of $A$.  A subtle technical
point here is to make the reduction robust with respect to the different orders
in which the controls can be applied.

\begin{proof}[Proof of Theorem~\ref{thm:undecidability}]
We reduce the vector reachability problem for invertible matrices to
the reachability problem for LTI systems.  

Let $A_1, \ldots, A_k \in \Q^{d \times d}$ be invertible matrices and
$\boldx, \boldy \in \Q^d$.  From these data we define an LTI system
$\mathcal{L}=(A,U)$ in dimension $D:=(k+1)d+k$.  We consider the state
space of $\mathcal{L}$ to be
\[ \underbrace{\R^d \oplus \cdots \oplus \R^d}_{k+1} \oplus \R^{k}
  \, ,\] that is, each state comprises a $(k+1)$-tuple of vectors in
$\mathbb{R}^d$ followed by a single vector in $\mathbb{R}^{k}$.

Matrix $A$ is a block diagonal matrix of dimension
$D\times D$, given by
\[ A := \diag{I_d, A_1, \ldots, A_k, I_k} \, . \]

For all $i \in \set{1,\ldots,k}$ define $V_i\subseteq \R^D$ by
\[ V_i := \set{\boldsymbol{0}}^{i-1} \times \set{(\boldz,-\boldz) :
    \boldz\in\R^d} \times \set{\boldsymbol{0}}^{k-i} \times
  \set{\bolde_i} \, , \] where $\bolde_i \in \R^{k}$ denotes $i$-th
coordinate vector and $\boldsymbol{0}$ denotes the zero vector in
$\R^d$.

We now define the set of controls $U\subseteq \R^D$  by
\begin{align*}
U:=V_1 +  \cdots + V_k \, .
\end{align*}
Given $i\in\set{1,\ldots,k}$, we think of $V_i$ as being
comprised of \emph{atomic controls}.  Such a control is determined by
the index $i$ and a vector $\boldz \in \R^d$.  Application of the
control subtracts $\boldz$ from the $i$-th block within the global
state and adds $\boldz$ to the $(i+1)$-st block.  Intuitively the
definition of $U$ as a sum of the $V_i$ allows to apply several atomic
controls at the same time.  Note that $U$ can be written as a union of 
affine subspaces.

Finally, we define the initial state to be $\bolds := (\boldx, \boldsymbol{0}, \ldots,
\boldsymbol{0}, \boldsymbol{0})$ and the target state to be $\boldt := \left(\boldsymbol{0},
\ldots, \boldsymbol{0}, \boldy, \boldsymbol{1} \right)$.

This completes the definition of the LTI system.  We now argue that
$\boldt$ is reachable from $\bolds$ if and only if there exist
$n_1, \ldots, n_k \in \mathbb{Z}$ such that
$\prod\limits_{i=1}^{k}A_{i}^{n_{i}} \boldx = \boldy$.  We divide the
argument into two claims.

\begin{claim}
Given integers $n_1,\ldots,n_k$, there exist non-negative integers
$t_1,\ldots,t_{k+1}$ such that
\begin{equation}
\begin{array}{rcl}
n_1 &=& t_2-t_1\\
       &\vdots & \\
n_k &=& t_{k+1}-t_k \, .
\end{array}
\label{eq:star}
\end{equation}
\label{claim1}
\end{claim}
\begin{proof}
  The proof is by induction on $k$.  The base case ($k=1$) is obvious.
  For the induction step, suppose we are given integers
  $n_1,\ldots,n_{k+1}$.  By induction we can find nonnegative integers
  $t_1,\ldots,t_{k+1}$ such that (\ref{eq:star}) holds.  We can
  moreover assume that $n_{k+1}+t_{k+1} \geq 0$ since (\ref{eq:star})
  will still hold if we translate all $t_1,\ldots,t_{k+1}$ by a common
  integer.  Hence we can define $t_{k+2}:=n_{k+1}+t_{k+1}$ and we have
  $n_{k+1}=t_{k+2}-t_{k+1}$.
\end{proof}

\begin{claim}
Vector $\boldt$ is reachable from $\bolds$ if and only if there exist
nonnegative integers $t_1,\ldots,t_{k+1}$ and vectors
$\boldz_1,\ldots,\boldz_{k+1} \in \mathbb{R}^d$ such that the
following equations hold:
\begin{equation}
\begin{array}{rcl}
\boldz_1&=& \boldx \\
\boldz_{i+1} &=& A_i^{t_{i+1}-t_i} \boldz_i \quad i\in\set{1,\ldots,k}\\
\boldz_{k+1} &=& \boldy \, .
\end{array}
\label{eq:star2}
\end{equation}
\label{claim2}
\end{claim}
\begin{proof}
  Suppose that $\boldt$ is reachable from $\bolds$.  Note that the
  final $k$ coordinates of $\bolds$ are all zero, while the
  corresponding coordinates of $\boldt$ are all one.  Since the block
  of matrix $A$ corresponding to these coordinates is the identity
  $I_{k}$, it follows that in going from $\bolds$ to $\boldt$ exactly
  one atomic control from each space $V_i$ was used for all $i\in
  \set{1,\ldots,k}$.  For each $i$ write $\boldz_i \in \R^d$ for the
    vector that determines the control in $V_i$ and let $t_i$ denote
    the number of steps after  which this control was applied.  Finally define
    $\boldz_{k+1}:=\boldy$ and let $t_{k+1}$ be the total number of steps
    going from $\bolds$ to $\boldt$.  We will show that with these definitions 
(\ref{eq:star2}) is satisfied.

Since the first block of $A$ is $I_d$, in order to reach a state of
$\mathcal{L}$ whose first block is $\boldsymbol{0}$ we must have
$\boldz_1=\boldx$.  In similar fashion, considering the $(i+1)$-st
block for $i=1,\ldots,k$, we have
\[A_i^{t_{k+1}-t_i} \boldz_i - A_i^{t_{k+1}-t_{i+1}}\boldz_{i+1} = 0\]
and hence, since $A_i$ is invertible,
$\boldz_{i+1} = A_i^{t_{i+1}-t_i} \boldz_i$.

Conversely, suppose that there exist nonnegative integers
$t_1,\ldots,t_{k+1}$ and vectors
$\boldz_1,\ldots,\boldz_{k+1} \in \R^d$ satisfying (\ref{eq:star2}).
Then $\boldt$ is reachable from $\bolds$ in $t_{k+1}$ steps by, for
all $i\in\set{1,\ldots,k}$, applying at time $t_i$ the atomic control
in $V_i$ that is determined by $\boldz_i$.  Indeed since we thereby
apply one atomic control for each $i\in\set{1,\ldots,k}$ we reach a
state with vector $\boldsymbol{1}$ in the final block.  The equations
in (\ref{eq:star2}) moreover guarantee that the reached state has
$\boldsymbol{0}$ in its first $k$ blocks and $\boldsymbol{y}$ in the
$(k+1)$-st block, i.e., the reached state is identical to $\boldt$.
% The first equation in
%(\ref{eq:star2}) ensures that this sequence of controls leads to
%$\boldsymbol{0}$ in the first block and likewise the remaining equations
%ensure that we end up with $\boldy$ in the 
%(\ref{eq:star2}) ensures that 
\end{proof}

We now complete the proof of Theorem~\ref{thm:undecidability} by combining the above
two claims.

Suppose that there exist integers $n_1,\ldots,n_{k}$ such that
$\prod\limits_{i=1}^k A_i^{n_{i}} \boldx = \boldy$. 
By Claim~\ref{claim1} there exists nonnegative integers
$t_1,\ldots,t_{k+1}$ satisfying
(\ref{eq:star}).
Then if we define $\boldz_i:= \prod_{j=1}^{i-1} A^{n_j} \boldx$ for
$i \in \set{1,\ldots,k+1}$ we have that Equation (\ref{eq:star2}) is
satisfied.  By Claim~\ref{claim2} it follows that $\boldt$ is reachable from
$\bolds$ in the LTI system $\mathcal L$.

Conversely suppose that $\boldt$ is reachable from $\bolds$.  By
Claim~\ref{claim2} the system (\ref{eq:star2}) has a solution.
Defining $n_i:=t_{i+1}-t_i$ for $i\in\set{1,\ldots,k}$ we have that
$\prod\limits_{i=1}^k A_i^{n_{i}} \boldx = \boldy$.

\end{proof}

\subsection{Positivity Hardness for Convex Control Sets}
Consider a sequence of integers
$\langle x_n : n \in \mathbb{N} \rangle$.  We say that such a sequence
satisfies a \emph{linear recurrence of order $d$} if there exist
$a_1,\ldots,a_d \in \mathbb{Z}$ such that
\[ x_n = \sum_{i=1}^d a_i x_{n-i}\] for all $n\geq d$.  Skolem's
Problem asks, given such a sequence (specified by a recurrence and its
initial values $x_0,\ldots,x_{d-1}$), whether $x_n=0$ for some $n$.
Likewise the Positivity Problem asks whether $x_n \geq 0$ for all $n$.
Decidability of Skolem's Problem and the Positivity Problem has been
open since the 1970s~\cite{SalomaaS78,RS94,Tao08,TUCS05}.  To date,
decidability of Skolem's Problem is known only for recurrences of
order at most 4~\cite{Ver85,MST84} and decidability of the Positivity
Problem is known only for recurrences of order at most 5~\cite{PP}.
There is a relatively straightforward reduction of Skolem's Problem to
the Positivity Problem (which does not preserve the order of
recurrences).

In this section we show that if the set of controls is a convex
polytope then the LTI reachability problem is as hard as the
Positivity Problem.  Instead of reducing from
the Positivity Problem directly, we give a reduction from the
\emph{Markov Reachability Problem}:
given a column-stochastic matrix $M \in \Q^{d \times d}$, determine
whether there exists $n \in \N$ such that
$(M^n)_{1,2} \geq \frac{1}{2}$.  A reduction from the Positivity
Problem to the Markov Reachability Problem has been given
in~\cite{MRP}.
\begin{theorem}
There is a reduction from the Positivity Problem to the reachability
problem for LTI systems whose sets of controls are compact convex
polytopes (with rational vertices).
\end{theorem}
\begin{proof}
We give a polynomial-time reduction from the Markov Reachability
Problem to the problem at hand.  Given a column-stochastic matrix
$M \in \Q^{d \times d}$, we define an LTI system comprising a matrix
$A=\diag{M, 0, 0,1} \in \Q^{(d+3)\times (d+3)}$ and a compact convex
polytope
\begin{equation*}
U = \left\{ \left(-\boldx, y, z,z \right) :
    \boldx \geq \boldsymbol{0}, 0 \leq y \leq x_1, \mbox{ and }
    \sum_{i = 1}^d x_i = z \leq 1 \right\} \subseteq \R^{d+3} \, .
\end{equation*}
The initial state is $\bolds = (\boldsymbol{e}_2,0,0,0) \in \Q^{d+3}$
and target state $\boldt = (\boldsymbol{0}, \frac{1}{2}, 1,1) \in \Q^{d+3}$.

We argue that $\boldt$ is reachable from $\bolds$ if and only if there
exists $n \in \N$ such that $(M^n)_{1,2} \ge \frac{1}{2}$.

First, suppose that there exists $n \in \N$ such that
$(M^n)_{1,2} \geq \frac{1}{2}$.  Consider the sequence of controls
$\boldu_{0} = \cdots = \boldsymbol{u}_{n-2} = \boldsymbol{0}$
and $\boldu_{n-1}=(-M^n\boldsymbol{e}_2,\frac{1}{2},1,1)\in U$.
This sequence steers $\bolds$ to $\boldt$.

On the other hand, suppose that there exists a sequence of controls
$\boldu_0,\ldots,\boldu_{n-1}$ controlling $\bolds$ to $\boldt$.
Since the $(d+2)$-nd and $(d+3)$-rd coordinates of $\boldt$ are equal,
noting that the matrix $A$ erases coordinate $d+2$ but not $d+3$, it
follows that $\boldu_{n-1}$ is the only non-zero control, that is,
$\boldu_{0} = \cdots = \boldu_{n-2} = \boldsymbol{0}$.  Therefore at
time $n-1$ the state is $(M^{n-1}\boldsymbol{e}_2, 0, 0, 0)$, and the
only way to reach $\boldt$ in the remaining step is to take
$\boldu_{n-1} =
(-M^{n-1}\boldsymbol{e}_2,\frac{1}{2},1,1)\in U$ (otherwise
one of the first $d$ coordinates will be non-zero).  But this is only
possible if $(M^n)_{1,2} \geq \frac{1}{2}$.

To conclude the proof, observe that the LTI reachability
instance $(\bolds, A, U, \boldt)$ can be constructed in polynomial
time from $M$.
\end{proof}

\subsection{Skolem Hardness}
\label{sec:app-skolem}

In this section we consider the case that the set of controls is the
union of an affine subspace and the origin.  We show that in this
case the LTI reachability problem is as hard as Skolem's Problem.
%Skolem's problem
%is a long-standing open problem in logic and number theory, whose
%decidability has not yet been determined.

We will work with the following matricial version of Skolem's Problem,
rather than the formulation in terms of linear recurrences.  The two
versions are easily seen to be interreducible~\cite{RS94}.
\begin{definition}
Given a matrix $M \in \Q^{d \times d}$, \emph{Skolem's problem}
consists in determining whether there exists a number $n \in \N$ such
that $(M^n)_{1,2} = 0$.
\end{definition}

We will now show the following result:
\begin{theorem}
There is a reduction from Skolem's Problem to the reachability problem
for LTI systems whose set of controls is the unions of an affine subspace
(with a basis of rational vectors) and the origin.
\end{theorem}

\begin{proof}
We give a (polynomial-time) reduction from the Skolem Problem to the
problem at hand.  Given a matrix $M \in \Q^{d \times d}$ we define the
matrix $A = \diag{M,2} \in \Q^{(d+1) \times (d+1)}$ and the set of
admissible controls $\calP = \set{0} \cup \set{(0,x_2,\ldots,x_d,1)
  \mid x_2,\ldots,x_d \in \R}$, as well as the source $\bolds =
(\boldsymbol{e}_2,0) \in \Q^{d+1}$ and target $\boldt = (0,\ldots,0, 1) \in
\Q^{d+1}$.

We argue that that $\boldt$ is reachable from $\bolds$
if and only if 
there exists $n \in \N$ such that $(M^n)_{1,2} = 0$.

First, suppose that there exists $n \in \N$ such that $(M^n)_{1,2} = 0$. 
Consider the sequence of controls given by
$\boldu_{0} = \cdots = \boldsymbol{u}_{n-2} = \boldsymbol{0}$
and
\[
\boldu_{n-1} = \left( 0, - (M^n)_{2,1}, - (M^n)_{3,1},\ldots, - (M^n)_{n,1}, 1 \right).
\]
This sequence steers $\bolds$ to $\boldt$.

On the other hand, suppose there exists a sequence of controls
$\boldu_0,\ldots,\boldu_{n-1}$ steering $\bolds$ to~$\boldt$.  The
very last dimension implies that a vector in the second set of
controls can be played only once, and in the last step. So the
sequence of controls is zero vectors except at the last step.  Looking
at the first coordinate this implies that $(M^n)_{1,2} = 0$.
\end{proof}

\section{Decidability of Reachability for Simple LTI Systems}
\label{sec:simple}
Define an LTI system $\mathcal{L}=(A,U)$ to be \emph{simple} if:
\begin{enumerate}
\item the set of controls $U\subseteq\R^d$ is a bounded convex
  polytope that contains $\boldzero$ in its relative interior (\emph{i.e.} $0\in B\cap \Span(U)\subseteq U$ for some open ball $B$ around
  $\boldzero$);
\item the spectral radius of $A$ is less than one;
\item some positive power of $A$ has exclusively real spectrum
\end{enumerate}
The rationale behind those assumptions will be explained in the following paragraphs.

We show decidability of the following problem: given a simple LTI
system in dimension $d$ and a bounded convex polytope
$Q\subseteq \R^d$, determine whether $Q$ is reachable from
$\boldzero$.  We call this the \emph{reachability problem for
  simple LTI} systems.
The set of states reachable from $\boldzero$ is
$A^*(U) := \bigcup_{m=0}^\infty \sum_{i=0}^m A^i(U)$.  Thus the
reachability problem for simple LTI systems is equivalent to asking whether
$A^*(U)$ meets $Q$.

It is clear that the reachability problem for LTI systems is semi-decidable.
Fixing $n\in\N$, the problem of whether $\sum_{i=0}^n A^i(U)$
meets a given polytope $Q$ can straightforwardly be cast as a linear
program.  Iterating over all $n\in\N$ we thus have a
semi-decision procedure for reachability.

In the rest of this section we describe a semi-decision procedure for
non-reachability.  The idea is to use a hyperplane that separates
$A^*(U)$ and $Q$ as a certificate of non-reachability, as illustrated in
Figure~\ref{fig:separating_hyperplane}. The key
technical step here is to show that it suffices to consider
hyperplanes whose normal vectors have algebraic entries.  To establish
this we consider the cone of all hyperplanes that separate $A^*(U)$
and $Q$ and show that the extremal elements of this cone are
algebraic.  This reasoning heavily relies on the spectral assumptions
about the matrix $A$ in the definition of simple LTI systems.

The main difficulty in certifying non-reachability is that the convex
set $A^*(U)$ is difficult to describe in general. In particular, it
can have infinitely many faces, even though the control polyhedron
only has finitely many faces, as illustrated in
Figure~\ref{fig:convex-infinitely-many-faces}. This is, however, not
the only difficulty. Indeed, one might get the impression from
Figure~\ref{fig:convex-infinitely-many-faces} that the boundary of
$A^*(U)$ consists solely of (possibly countably many) facets, that is
faces of dimension $d-1$.  If that were true then one could one could
separate $A^*(U)$ from any point not in $A^*(U)$ by a hyperplane that
supports one of the facets of $A^*(U)$.  Moreover since 
for a simple LTI system any facet 
of $A^*(U)$ has 
a supporting hyperplane with a rational normal vector, this would mean
that for points outside $A^*(U)$ one can always find separating
hyperplanes with rational coefficients.
%write\footnote{Recall that any closed convex set can be written as the
%  intersection of its closed supporting half-spaces.}  the closure of
%$A^*(U)$ as the intersection of the supporting half-spaces of its
%facets.  Since any facet has a supporting hyperplane with a
%rational normal vector we would obtain a ``rational'' description
%of the reachable set. 
Unfortunately however there are cases, as illustrated in
Figure~\ref{fig:convex-lower-dim-face}, where part of the boundary of
the reachable set does not belong to any facet, but rather to some
lower-dimensional face. Such faces, by definition, do not usually have
a unique supporting hyperplane and it is not \emph{a priori} clear
that they admit a supporting hyperplane with algebraic coefficients.
The main technical result of this section
(Proposition~\ref{prop:one-dim}) shows that such algebraic supporting
hyperplanes can always be found.  It follows that for any point not in
$A^*(U)$ there is always a separating hyperplane whose normal has
algebraic-number coefficients.

\begin{figure*}
    \begin{center}
    \begin{tikzpicture}[
        fillcolor/.style={
            draw=none, pattern color=darkgreen, pattern=crosshatch dots
        },
        legend style/.style={
            draw=none, fill=white, inner sep=1.5pt
        },
        ptcolor/.style={draw=blue,circle,inner sep=\fracszpt,fill=blue},
        edgecolor/.style={blue,line width=\fracszedge}, normcolor/.style=red,
        qpoly/.style={draw=red,line width=\fracszedge,pattern color=red, pattern=crosshatch dots},
        qsep/.style={line width=0.5pt,gray},
        qtau/.style={line width=1pt,gray,->},
        conedge/.style={black,dashed}, coneinfedge/.style={black,dashed},
        conefill/.style={fill=gray, fill opacity=0.15},
        scale=1.5]
        \newcommand{\fracszpt}{0.7pt}
        \newcommand{\fracszedge}{0.7pt}
        \coordinate (vc0) at (-1, 0) {};
        \coordinate (vc1) at (0, 1) {};
        \coordinate (vc2) at (1, 0) {};
        \coordinate (vc3) at (1, -1/2) {};
        \coordinate (vc4) at (0, -2/3) {};
        \draw[edgecolor] (vc0) -- (vc1) -- (vc2) -- (vc3) -- (vc4) -- cycle;
        \node[ptcolor] at (vc0) {};
        \node[ptcolor] at (vc1) {};
        \node[ptcolor] at (vc2) {};
        \node[ptcolor] at (vc3) {};
        \node[ptcolor] at (vc4) {};
        \fill[fillcolor]  (vc0) -- (vc1) -- (vc2) -- (vc3) -- (vc4) --cycle;
        \draw (0,0) node[legend style] {$A^*(U)$};

        \draw[qpoly] (-1/2,1.2) -- (1/2,1.2) -- (1/2,1.8) -- (-1/2,1.8) -- cycle;
        \draw (0,1.5) node[legend style] {$Q$};
        \draw[qsep] (-1.3,1) -- (0.9,1) node[above, pos=0.2] {$H$};

        \draw[qpoly] (1,1/2) -- (1.8,2/3) -- (1.8,-1/2) -- (1, -1/2) -- cycle;
        \draw (1.4,0) node[legend style] {$Q$};
        \draw[qsep] (1,0.9) -- (1,-1.2) node[right, pos=0.8] {$H$};

        \draw[qpoly] (-1.5,0.1) -- (-1/2,-2/3) -- (-1.5,-2/3) -- cycle;
        \draw (-1.2,-0.45) node[legend style] {$Q$};
        \draw[qsep] (-1.5,1/3) -- (1/2,-1) node[above right, pos=0.1] {$H$};

        \begin{scope}[shift={(3,0)}]
        \coordinate (a) at (-0.1,1.1) {};
        \coordinate (vert) at (1.5,0) {};
        \coordinate (b) at (-0.1,-1.1) {};
        \coordinate (zero) at (0,0) {};
        \node[ptcolor] at (vert) {};
        \draw[edgecolor] (a) -- (vert) -- (b);
        \draw[fillcolor] (a) -- (vert) -- (b) -- cycle;
        \node[black,fill,circle,inner sep=1pt] at (zero) {} node[above] {$\boldzero$};
        \draw (0.8,0.05) node[legend style] {$A^*(U)$};

        \coordinate (qtl) at (2, 1.5) {};
        \coordinate (qbl) at (2.2, -1) {};
        \coordinate (qbr) at (3.5, -2/3) {};
        \coordinate (qtr) at (3.5, 2/3) {};
        \draw[qpoly] (qtl) -- (qbl) -- (qbr) -- (qtr) -- cycle;
        \draw (2.7,0) node[legend style] {$Q$};

        \newcommand{\tausz}{0.25cm}
        \draw[qsep] (1.5, 1.5) -- (1.5,-1.5);
        \draw[qtau] (1.5, 0) -- ++(\tausz,0) node[right] {$\boldtau$};
        \draw[qsep] ($(vert)!-1.5!(qbl)$) -- ($(vert)!1.1!(qbl)$);
        \coordinate (intersectbl) at ($(vert)!(zero)!(qbl)$) {};
        \draw[qtau] (intersectbl) -- ++($(zero)!\tausz!(intersectbl)$) node[above] {$\boldtau$};
        \draw[qsep] ($(vert)!-1!(qtl)$) -- ($(vert)!1.1!(qtl)$);
        \coordinate (intersecttl) at ($(vert)!(zero)!(qtl)$) {};
        \draw[qtau] (intersecttl) -- ++($(zero)!\tausz!(intersecttl)$) node[below] {$\boldtau$};

        \newcommand{\conesz}{4.3cm}
        \coordinate (conevertbl) at ($(zero)!\conesz!(intersectbl)$) {};
        \coordinate (coneverttl) at ($(zero)!\conesz!(intersecttl)$) {};
        \draw[conedge] (zero) -- (conevertbl);
        \draw[conedge] (zero) -- (coneverttl);
        \draw[coneinfedge] (conevertbl) edge[bend left=30] (coneverttl);
        % note: the following is bit tricky, we need edge[] to do the bend, but edge[] in draw/fill
        % create a new path that is unaffected by the style, thus we would need to use
        % every edge/.style but it still does not solve the problem because the resulting path
        % without the edge is incomplete and does not fill probably. Better use the lower-level
        % path directive
        \path[conefill] (zero) to (conevertbl) to[bend left=30] (coneverttl);
        \node[gray] at (3.2,1.5) {$\CONE(A^*(U),Q)$};
        \end{scope}
    \end{tikzpicture}
    \end{center}
    \caption{Example of separating hyperplanes $H$ between the reachable set $A^*(U)$ and various
        targets polytopes $Q$. If $Q$ and $A^*(U)$ are disjoint, we can always separate them by a supporting
        hyperplane of $A^\infty(U)$, the closure of $A^*(U)$. The set $\CONE(A^*(U),Q)$
        of all separators is crucial to understanding the situation. Note that the direction of such a
        supporting hyperplane may not always correspond to a facet of $A^\infty(U)$ or $Q$,
        in particular see Figure~\ref{fig:convex-lower-dim-face}.
        \label{fig:separating_hyperplane}}
\end{figure*}
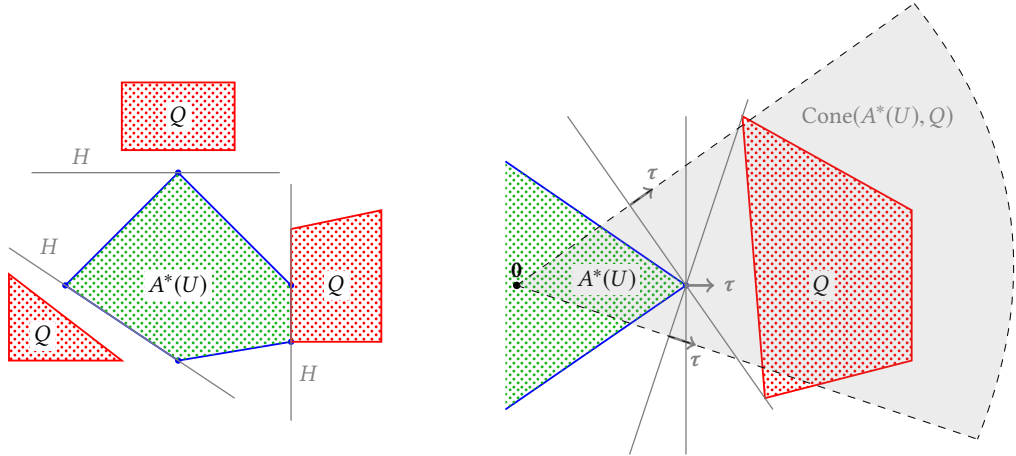

\begin{figure*}
    \begin{center}
    \begin{tikzpicture}[
            fillcolor/.style={
                draw=none, pattern color=darkgreen, pattern=crosshatch dots
            },
            legend style/.style={
                draw=none, fill=white, inner sep=1.5pt
            },
            ptcolor/.style={draw=blue,circle,inner sep=\fracszpt,fill=blue},
            edgecolor/.style={blue,line width=\fracszedge},
            normcolor/.style={red,line width=\fracsznorm,->},
            scale=0.75]
        \newcommand{\fracscale}{1}
        \newcommand{\fracszpt}{0.7pt}
        \newcommand{\fracszedge}{0.7pt}
        \newcommand{\fracsznorm}{0.7pt}
        \newcommand{\fracnormlen}{1}
        \newcommand{\fraccoord}{0pt}
        \coordinate (vc0) at (-2.99999993031*\fracscale,-2.99543268348*\fracscale) {};
\coordinate (vc1) at (-2.99999979092*\fracscale,-2.99086536696*\fracscale) {};
\coordinate (vc2) at (-2.99999937277*\fracscale,-2.98401439218*\fracscale) {};
\coordinate (vc3) at (-2.99999811832*\fracscale,-2.97373793*\fracscale) {};
\coordinate (vc4) at (-2.99999435497*\fracscale,-2.95832323675*\fracscale) {};
\coordinate (vc5) at (-2.99998306491*\fracscale,-2.93520119686*\fracscale) {};
\coordinate (vc6) at (-2.99994919474*\fracscale,-2.90051813703*\fracscale) {};
\coordinate (vc7) at (-2.99984758421*\fracscale,-2.84849354728*\fracscale) {};
\coordinate (vc8) at (-2.99954275263*\fracscale,-2.77045666266*\fracscale) {};
\coordinate (vc9) at (-2.99862825789*\fracscale,-2.65340133573*\fracscale) {};
\coordinate (vc10) at (-2.99588477366*\fracscale,-2.47781834533*\fracscale) {};
\coordinate (vc11) at (-2.98765432099*\fracscale,-2.21444385973*\fracscale) {};
\coordinate (vc12) at (-2.96296296296*\fracscale,-1.81938213134*\fracscale) {};
\coordinate (vc13) at (-2.88888888889*\fracscale,-1.22678953874*\fracscale) {};
\coordinate (vc14) at (-2.66666666667*\fracscale,-0.337900649854*\fracscale) {};
\coordinate (vc15) at (2.99999993031*\fracscale,2.99543268348*\fracscale) {};
\coordinate (vc16) at (-2.0*\fracscale,0.995432683479*\fracscale) {};
\coordinate (vc17) at (2.99999979092*\fracscale,2.99086536696*\fracscale) {};
\coordinate (vc18) at (2.99999937277*\fracscale,2.98401439218*\fracscale) {};
\coordinate (vc19) at (2.99999811832*\fracscale,2.97373793*\fracscale) {};
\coordinate (vc20) at (2.99999435497*\fracscale,2.95832323675*\fracscale) {};
\coordinate (vc21) at (2.99998306491*\fracscale,2.93520119686*\fracscale) {};
\coordinate (vc22) at (2.99994919474*\fracscale,2.90051813703*\fracscale) {};
\coordinate (vc23) at (2.99984758421*\fracscale,2.84849354728*\fracscale) {};
\coordinate (vc24) at (2.99954275263*\fracscale,2.77045666266*\fracscale) {};
\coordinate (vc25) at (2.99862825789*\fracscale,2.65340133573*\fracscale) {};
\coordinate (vc26) at (2.99588477366*\fracscale,2.47781834533*\fracscale) {};
\coordinate (vc27) at (2.98765432099*\fracscale,2.21444385973*\fracscale) {};
\coordinate (vc28) at (2.96296296296*\fracscale,1.81938213134*\fracscale) {};
\coordinate (vc29) at (2.88888888889*\fracscale,1.22678953874*\fracscale) {};
\coordinate (vc30) at (2.66666666667*\fracscale,0.337900649854*\fracscale) {};
\coordinate (vc31) at (2.0*\fracscale,-0.995432683479*\fracscale) {};
\coordinate (vc32) at (0.0*\fracscale,2.99543268348*\fracscale) {};
\coordinate (vc33) at (0.0*\fracscale,-2.99543268348*\fracscale) {};
\coordinate (barycenter) at (0.0,0.0);
\fill[fillcolor]  (vc15) -- (vc17) -- (vc18) -- (vc19) -- (vc20) -- (vc21) -- (vc22) -- (vc23) -- (vc24) -- (vc25) -- (vc26) -- (vc27) -- (vc28) -- (vc29) -- (vc30) -- (vc31) -- (vc33) -- (vc0) -- (vc1) -- (vc2) -- (vc3) -- (vc4) -- (vc5) -- (vc6) -- (vc7) -- (vc8) -- (vc9) -- (vc10) -- (vc11) -- (vc12) -- (vc13) -- (vc14) -- (vc16) -- (vc32) --cycle;
\draw[edgecolor,line width=\fracszedge]  (vc15) -- (vc17) -- (vc18) -- (vc19) -- (vc20) -- (vc21) -- (vc22) -- (vc23) -- (vc24) -- (vc25) -- (vc26) -- (vc27) -- (vc28) -- (vc29) -- (vc30) -- (vc31) -- (vc33) -- (vc0) -- (vc1) -- (vc2) -- (vc3) -- (vc4) -- (vc5) -- (vc6) -- (vc7) -- (vc8) -- (vc9) -- (vc10) -- (vc11) -- (vc12) -- (vc13) -- (vc14) -- (vc16) -- (vc32) --cycle;
\draw[normcolor] (-2.99999986062*\fracscale,-2.99314902522*\fracscale) -- ++(-0.999999999534*\fracnormlen,3.05175781108e-05*\fracnormlen);
\draw[normcolor] (-2.99999958185*\fracscale,-2.98743987957*\fracscale) -- ++(-0.999999998137*\fracnormlen,6.10351561363e-05*\fracnormlen);
\draw[normcolor] (-2.99999874555*\fracscale,-2.97887616109*\fracscale) -- ++(-0.999999992549*\fracnormlen,0.000122070311591*\fracnormlen);
\draw[normcolor] (-2.99999623665*\fracscale,-2.96603058337*\fracscale) -- ++(-0.999999970198*\fracnormlen,0.000244140617724*\fracnormlen);
\draw[normcolor] (-2.99998870994*\fracscale,-2.9467622168*\fracscale) -- ++(-0.999999880791*\fracnormlen,0.000488281191792*\fracnormlen);
\draw[normcolor] (-2.99996612982*\fracscale,-2.91785966694*\fracscale) -- ++(-0.999999523163*\fracnormlen,0.000976562034339*\fracnormlen);
\draw[normcolor] (-2.99989838947*\fracscale,-2.87450584215*\fracscale) -- ++(-0.999998092657*\fracnormlen,0.00195312127472*\fracnormlen);
\draw[normcolor] (-2.99969516842*\fracscale,-2.80947510497*\fracscale) -- ++(-0.999992370693*\fracnormlen,0.00390622019802*\fracnormlen);
\draw[normcolor] (-2.99908550526*\fracscale,-2.71192899919*\fracscale) -- ++(-0.999969483819*\fracnormlen,0.00781226159233*\fracnormlen);
\draw[normcolor] (-2.99725651578*\fracscale,-2.56560984053*\fracscale) -- ++(-0.999877952035*\fracnormlen,0.0156230930005*\fracnormlen);
\draw[normcolor] (-2.99176954733*\fracscale,-2.34613110253*\fracscale) -- ++(-0.999512076087*\fracnormlen,0.0312347523777*\fracnormlen);
\draw[normcolor] (-2.97530864198*\fracscale,-2.01691299553*\fracscale) -- ++(-0.998052578483*\fracnormlen,0.0623782861552*\fracnormlen);
\draw[normcolor] (-2.92592592593*\fracscale,-1.52308583504*\fracscale) -- ++(-0.992277876714*\fracnormlen,0.124034734589*\fracnormlen);
\draw[normcolor] (-2.77777777778*\fracscale,-0.782345094299*\fracscale) -- ++(-0.970142500145*\fracnormlen,0.242535625036*\fracnormlen);
\draw[normcolor] (-2.33333333333*\fracscale,0.328766016812*\fracscale) -- ++(-0.894427191*\fracnormlen,0.4472135955*\fracnormlen);
\draw[normcolor] (2.99999986062*\fracscale,2.99314902522*\fracscale) -- ++(0.999999999534*\fracnormlen,-3.05175781108e-05*\fracnormlen);
\draw[normcolor] (2.99999958185*\fracscale,2.98743987957*\fracscale) -- ++(0.999999998137*\fracnormlen,-6.10351561363e-05*\fracnormlen);
\draw[normcolor] (2.99999874555*\fracscale,2.97887616109*\fracscale) -- ++(0.999999992549*\fracnormlen,-0.000122070311591*\fracnormlen);
\draw[normcolor] (2.99999623665*\fracscale,2.96603058337*\fracscale) -- ++(0.999999970198*\fracnormlen,-0.000244140617724*\fracnormlen);
\draw[normcolor] (2.99998870994*\fracscale,2.9467622168*\fracscale) -- ++(0.999999880791*\fracnormlen,-0.000488281191792*\fracnormlen);
\draw[normcolor] (2.99996612982*\fracscale,2.91785966694*\fracscale) -- ++(0.999999523163*\fracnormlen,-0.000976562034339*\fracnormlen);
\draw[normcolor] (2.99989838947*\fracscale,2.87450584215*\fracscale) -- ++(0.999998092657*\fracnormlen,-0.00195312127472*\fracnormlen);
\draw[normcolor] (2.99969516842*\fracscale,2.80947510497*\fracscale) -- ++(0.999992370693*\fracnormlen,-0.00390622019802*\fracnormlen);
\draw[normcolor] (2.99908550526*\fracscale,2.71192899919*\fracscale) -- ++(0.999969483819*\fracnormlen,-0.00781226159233*\fracnormlen);
\draw[normcolor] (2.99725651578*\fracscale,2.56560984053*\fracscale) -- ++(0.999877952035*\fracnormlen,-0.0156230930005*\fracnormlen);
\draw[normcolor] (2.99176954733*\fracscale,2.34613110253*\fracscale) -- ++(0.999512076087*\fracnormlen,-0.0312347523777*\fracnormlen);
\draw[normcolor] (2.97530864198*\fracscale,2.01691299553*\fracscale) -- ++(0.998052578483*\fracnormlen,-0.0623782861552*\fracnormlen);
\draw[normcolor] (2.92592592593*\fracscale,1.52308583504*\fracscale) -- ++(0.992277876714*\fracnormlen,-0.124034734589*\fracnormlen);
\draw[normcolor] (2.77777777778*\fracscale,0.782345094299*\fracscale) -- ++(0.970142500145*\fracnormlen,-0.242535625036*\fracnormlen);
\draw[normcolor] (2.33333333333*\fracscale,-0.328766016812*\fracscale) -- ++(0.894427191*\fracnormlen,-0.4472135955*\fracnormlen);
\draw[normcolor] (1.49999996515*\fracscale,2.99543268348*\fracscale) -- ++(0.0*\fracnormlen,1.0*\fracnormlen);
\draw[normcolor] (-1.0*\fracscale,1.99543268348*\fracscale) -- ++(-0.707106781187*\fracnormlen,0.707106781187*\fracnormlen);
\draw[normcolor] (-1.49999996515*\fracscale,-2.99543268348*\fracscale) -- ++(0.0*\fracnormlen,-1.0*\fracnormlen);
\draw[normcolor] (1.0*\fracscale,-1.99543268348*\fracscale) -- ++(0.707106781187*\fracnormlen,-0.707106781187*\fracnormlen);
\node[ptcolor] (v0) at (vc0) {};
\node[ptcolor] (v1) at (vc1) {};
\node[ptcolor] (v2) at (vc2) {};
\node[ptcolor] (v3) at (vc3) {};
\node[ptcolor] (v4) at (vc4) {};
\node[ptcolor] (v5) at (vc5) {};
\node[ptcolor] (v6) at (vc6) {};
\node[ptcolor] (v7) at (vc7) {};
\node[ptcolor] (v8) at (vc8) {};
\node[ptcolor] (v9) at (vc9) {};
\node[ptcolor] (v10) at (vc10) {};
\node[ptcolor] (v11) at (vc11) {};
\node[ptcolor] (v12) at (vc12) {};
\node[ptcolor] (v13) at (vc13) {};
\node[ptcolor] (v14) at (vc14) {};
\node[ptcolor] (v15) at (vc15) {};
\node[ptcolor] (v16) at (vc16) {};
\node[ptcolor] (v17) at (vc17) {};
\node[ptcolor] (v18) at (vc18) {};
\node[ptcolor] (v19) at (vc19) {};
\node[ptcolor] (v20) at (vc20) {};
\node[ptcolor] (v21) at (vc21) {};
\node[ptcolor] (v22) at (vc22) {};
\node[ptcolor] (v23) at (vc23) {};
\node[ptcolor] (v24) at (vc24) {};
\node[ptcolor] (v25) at (vc25) {};
\node[ptcolor] (v26) at (vc26) {};
\node[ptcolor] (v27) at (vc27) {};
\node[ptcolor] (v28) at (vc28) {};
\node[ptcolor] (v29) at (vc29) {};
\node[ptcolor] (v30) at (vc30) {};
\node[ptcolor] (v31) at (vc31) {};
\node[ptcolor] (v32) at (vc32) {};
\node[ptcolor] (v33) at (vc33) {};
\draw (barycenter) node[legend style] {$A^*(U)$};
        \draw (7,2) node {$A=\begin{bmatrix}\tfrac{1}{3}&0\\0&\tfrac{2}{3}\end{bmatrix}$};
        \begin{scope}[shift={(7,-2)}]
        \input{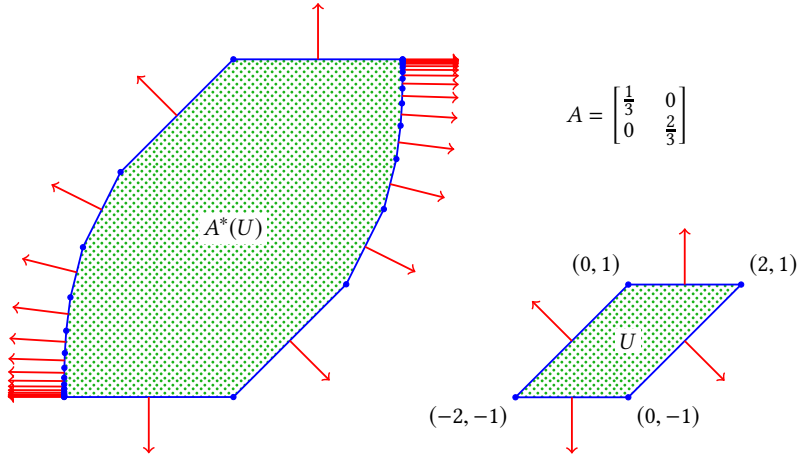}
        \end{scope}
    \end{tikzpicture}
    \end{center}
    \caption{Example of a simple LTI system where the reachable set $A^*(U)$ is a convex set with infinitely
        many faces. Note that since $A^*(U)$ is open, we represented its closure $A^\infty(U)$.
        \label{fig:convex-infinitely-many-faces}}
\end{figure*}

\begin{figure*}
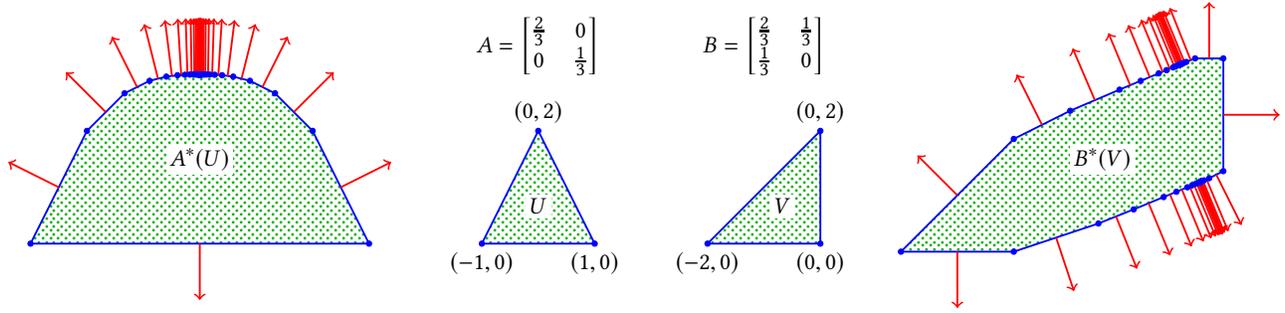

    \begin{center}
    \begin{tikzpicture}[
            fillcolor/.style={
                draw=none, pattern color=darkgreen, pattern=crosshatch dots
            },
            legend style/.style={
                draw=none, fill=white, inner sep=1.5pt
            },
            ptcolor/.style={draw=blue,circle,inner sep=\fracszpt,fill=blue},
            edgecolor/.style={blue,line width=\fracszedge},
            normcolor/.style={red,line width=\fracsznorm,->},
            scale=0.75]
        \newcommand{\fracscale}{1}
        \newcommand{\fracszpt}{0.7pt}
        \newcommand{\fracszedge}{0.7pt}
        \newcommand{\fracsznorm}{0.7pt}
        \newcommand{\fracnormlen}{1}
        \newcommand{\fraccoord}{0pt}
        \begin{scope}
            \input{lti-lowerdim-faces-reach.tex}
            \draw (6,3.5) node {$A=\begin{bmatrix}\tfrac{2}{3}&0\\0&\tfrac{1}{3}\end{bmatrix}$};
            \begin{scope}[shift={(6,0)}]
            % remark: I manually modified the file from the generated code to adjust the placement
            % of the vertices coordinates
            \input{lti-lowerdim-faces-U.tex}
            \end{scope}
        \end{scope}
        \begin{scope}[shift={(11,0)}]
            \input{lti-lowerdim2-faces-U.tex}
            \draw (-1,3.5) node {$B=\begin{bmatrix}\tfrac{2}{3}&\tfrac{1}{3}\\\tfrac{1}{3}&0\end{bmatrix}$};
            \begin{scope}[shift={(6,1)}]
            % remark: I manually modified the file from the generated code to adjust the placement
            % of the vertices coordinates
            \input{lti-lowerdim2-faces-reach.tex}
            \end{scope}
        \end{scope}
    \end{tikzpicture}
    \end{center}
    \caption{Example of simple LTI system in which an extreme point does not belong to a facet: the point $x=(0,3)$
        ``at the top'' of $A^*(U)$, on the left, does not belong to any of the infinitely many edges that
        ``converge'' to it. In particular $x$ is extreme and exposed but it is not a vertex.
        It can be separated from $A^*(U)$ by the hyperplane $y\geqslant 3$ which has rational coefficients.
        However this is not always the case: $B^*(V)$ has two extreme points that do not belong to any
        facet and have rational coordinates, but whose (unique) separating hyperplane requires the use  of algebraic irrationals. In particular, the normal to the separating hyperplanes is,
        up to sign, one the eigenvectors of $B$.
        \label{fig:convex-lower-dim-face}}
\end{figure*}

The following proposition identifies some simplifying assumptions that
can be made without loss of generality for analysing reachability.

\begin{proposition}
  The reachability problem for a simple LTI system $\mathcal{L}=(A,U)$ can be
  reduced to the special case in which it is assumed that all
  eigenvalues of $A$ are real and strictly positive and that $A^*(U)$ is
  full dimensional.
\label{prop:WLOG}
\end{proposition}
\begin{proof}%[Proof of Proposition~\ref{prop:WLOG}]
  We first reduce to the case in which all eigenvalues of $A$ are real
  and nonnegative.  Let $M$ be the least common multiple of the orders\footnote{A root of unity of order $r$
    has degree $\varphi(r) \geq \frac{r}{283 \log \log r}$ as an
    algebraic number, where $\varphi(r)$ is Euler's function (see,
    e.g.,~\cite{TUCS05}).  From this fact it is straightforward to
    compute an upper bound on $M$.}
  of the roots of unity $\frac{\lambda}{|\lambda|}$ for $\lambda$ a
  non-zero eigenvalue of $A$.  Then $A^M$ has spectrum included
  in the nonnegative real line.  But $Q \subseteq\R^d$ is
  reachable in the original LTI iff it is reachable in the LTI with
  transition matrix $A^M$ and set of controls
  $\sum_{i=0}^{M-1} A^i(U)$.

  We next give a reduction to the case that $A$ is invertible.  Write
  $\R^d = V_0\oplus V_1$, where $V_0,V_1\subseteq \R^d$
  are $A$-invariant subspaces such that $A$ is nilpotent on $V_0$ (of
  index at most $d$) and invertible on $V_1$.  
Then for all
  $n\geq d$ and $\boldu_0,\ldots,\boldu_n \in U$ we have
\begin{gather}
 \sum_{i=0}^n A^i \boldu_i = \sum_{i=0}^{d-1} A^i\boldu_i 
+ \sum_{i=0}^{n-d} A^iA^d\boldu_{i+d} \, . 
\label{eq:split}
\end{gather}
Suppose that $V_1$ has dimension $d_1$.  Pick a basis of $V_1$ and
with respect to this basis let $U'\subseteq \R^{d_1}$
represent the set $A^d(U) \subseteq V_1$, let
$Q'\subseteq \R^{d_1}$ represent
$(Q - \sum_{i=0}^{d-1} A^i (U))\cap V_1$, and let $A'$ represent the
linear transformation on $V_1$ induced by $A$.  From (\ref{eq:split})
it is clear that $Q$ is reachable in $\mathcal{L}$ if and only if $Q'$
is reachable in the LTI system $\mathcal{L}':=(A',U')$, in which $A'$
is an invertible matrix.

Now suppose that $\mathcal{L}=(A,U)$ a simple LTI system in dimension
$d$ with $A$ invertible.  Let $V\subseteq\R^d$ be the least
$A$-invariant subspace that contains $U$.  Then restricting $A$ to $V$
one obtains an LTI system in which $A^*(U)$ is full dimensional (with
the restriction of $A$ to $V$ remaining invertible).

\end{proof}

In the rest of this section we will assume that all simple LTI systems
are such that $A^*(U)$ is full dimensional and all eigenvalues of $A$
are strictly positive.

\begin{proposition}\label{prop:closure}
  Given a simple LTI system $\mathcal{L}=(A,U)$, the reachable set
  $A^*(U)$ is a convex open subset of $\R^d$ whose closure is
  $A^\infty(U) := \sum_{i=0}^\infty A^i(U)$.
\end{proposition}
\begin{proof}
Since $\boldzero\in U$ it holds that 
\[ A^0(U) \, \subseteq \,  A^0(U)+A^1(U) \, \subseteq A^0(U)+A^1(U)+A^2(U) \, \subseteq \, \cdots \]
is an increasing chain of convex subsets of $\R^d$ and hence the union of the chain
$A^*(U)$ is convex.

We next show that $A^*(U)$ is open.  Now the increasing sequence of
vector spaces $V_0 \subseteq V_1 \subseteq \ldots$, where
$V_j:=\Span\left(\sum_{i=0}^j A^i(U)\right)$, stabilises in at most
$d$ steps.  Thus from the assumption that $A^*(U)$ is full dimensional
we conclude that $\sum_{i=0}^{d-1} A^i(U)$ is already full
dimensional.  Furthermore, by Assumption 1 in the definition of Simple
LTI system, we in fact have that $\sum_{i=0}^{d-1} A^i(U)$ contains a
full-dimensional subset that is symmetric around $\boldsymbol{0}$ and
hence contains $\boldsymbol{0}$ in its interior.  Now consider a
typical element $\sum_{i=0}^n A^i \boldu_i \in A^*(U)$.  Since $A$ is
invertible we have that
$A^{n+1}(\sum_{i=0}^{d-1} A^i(U)) = \sum_{i={n+1}}^{d+n} A^i(U)$ is
also full dimensional and contains $\boldzero$ in its interior.  We
conclude that $\sum_{i=0}^n A^i \boldu_i$ lies in the interior of
$\sum_{i=0}^n A^i \boldu_i + \sum_{i={n+1}}^{d+n} A^i(U)$, which
itself is contained in the interior of $A^*(U)$.

From the fact that the spectral radius of $A$ is strictly less than
one it easily follows that $A^*(U)$ is dense in $A^\infty(U)$, so it
remains to observe that $A^\infty(U)$ is closed.  But $A^\infty(U)$ is
a fixed point of the contractive self map $F: X \mapsto A(X)+U$ on the
metric space of all bounded subsets of $\R^d$ under the
Hausdorff metric.  Such a self-map has a unique fixed point.
Moreover, since the collection of compact subsets of $\R^d$ is
complete under the Hausdorff metric and is preserved by $F$, we
conclude that $A^\infty(U)$ is compact and thus closed.
\end{proof}

The following is a (version of a) classical result of convex analysis:
\begin{theorem}[Theorem of the Separating Hyperplane]\label{thm:separating}
  Let $C$ and $D$ be compact convex subsets of $\R^d$.  Then
  $\mathrm{int}(C) \cap D=\emptyset$ if and only if there exists
  $\boldtau\in\R^d$ and $b\in\R$ such that
  $\Angle{\boldx,\boldtau}\leq b$ for all
  $\boldx \in C$ and
  $\Angle{\boldx,\boldtau}\geq b$ for all
  $\boldx\in D$.
\end{theorem}

It follows from Proposition~\ref{prop:closure} and
Theorem~\ref{thm:separating} that $A^*(U)$ and $Q$ are disjoint if and
only if there exists $\boldtau\in\R^d$ and
$b\in\R$ such that
$\Angle{\boldx,\boldtau}\leq b$ for all
$\boldx \in A^\infty(U)$ and
$\Angle{\boldx,\boldtau}\geq b$ for all
$\boldx\in Q$.  This motivates us to define the \emph{cone of
  separators} of $A^*(U)$ and $Q$ to be
\[ \CONE(A^\infty(U),Q) := \left\{ \boldtau \in \R^d : \forall
    \boldu \in A^\infty(U) \, \forall \boldv \in Q , \Angle{\boldu,\boldtau} \leq
    \Angle{\boldv,\boldtau} \right\} \, . \] It is straightforward to verify
that $\CONE(A^\infty(U),Q)$ is a topologically closed cone in
$\R^d$.  Moreover by the assumption that $A^\infty(U)$ is full
dimensional we have that $\CONE(A^\infty(U),Q)$ is a \emph{pointed
  cone}, that is, for all $\boldtau\in\R^d$ if both
$\boldtau,-\boldtau \in \CONE(A^\infty(U),Q)$
then $\boldtau=\boldzero$. See Figure~\ref{fig:separating_hyperplane} for
graphical representation of the cone.

The main goal of Propositions~\ref{prop:rational-supporting-hyperplane}--\ref{prop:one-dim}
is to show the following lemma. Intuitively, if the boundary of the reachable set
is a face, then there is only one possible direction for a separator (the normal to the
face) and it is algebraic because the face has an algebraic description (Proposition~\ref{prop:rational-supporting-hyperplane}).
If it is a lower dimensional face, then there is a cone of possible directions
and we need to show that it contains an algebraic one. The idea is to choose
a vector satisfying particular conditions in the hope that those conditions
allow us to recover unicity of the direction (and algebraicity by writing
equations with algebraic entries). One such condition is to consider
an extremal vector $\tau$ of the cone. Intuitively, a vector is extremal because
of two possible reasons: either because the supporting hyperplane is ``blocked'' by $Q$
(see Figure~\ref{fig:separating_hyperplane}), then we get (algebraic) equations from $Q$.
Or because changing $\tau$ would change the face that the hyperplane is supporting
(see Figure~\ref{fig:dominating_set}), then we get that $\tau$ is ``dominating''
in some sense, and this gives us further equations. The crux of of the proof
is Proposition~\ref{prop:one-dim} showing that when all factors are considered, there
is essentially a unique separating direction and therefore it must be algebraic.

\begin{lemma}\label{lem:algebraic_separator}
    If there is a separator of $A^*(U)$ and $Q$ then there is an algebraic separator.
\end{lemma}

Given $\boldtau\in\R^d$ and a closed set $C\subseteq\R^d$, define
\begin{eqnarray*}
\MAX_{\boldtau}(C)&:=& \{ \boldv\in C:\forall \boldw\in C, \Angle{\boldv,\boldtau}\geq\Angle{\boldw,\boldtau}\}\\
\MIN_{\boldtau}(C)&:=& \{ \boldv\in C:\forall \boldw\in C, \Angle{\boldv,\boldtau}\leq\Angle{\boldw,\boldtau}\} \, .
\end{eqnarray*}

For fixed $\boldtau\in\R^d$,
note that $\left\langle \sum_{i=0}^\infty A^i\boldu_i ,\boldtau\right\rangle$
is maximised for $\boldu_0,\boldu_1,\ldots \in U$ if and only if each individual
inner product $\left\langle A^i\boldu_i ,\boldtau\right\rangle$ is maximised.
In other words,
\begin{gather} \MAX_{\boldtau}(A^\infty(U)) = \MAX_{\boldtau}(A^0(U))+
\MAX_{\boldtau}(A^1(U))+
\MAX_{\boldtau}(A^2(U))+  \cdots 
\label{eq:local-max}
\end{gather}
Given $S\subseteq\R^d$, define
\[ \Aff_0(S) = \left\{ \sum_{i=1}^m \lambda_i \boldu_i : \sum_{i=1}^m
    \lambda_i=0 \mbox{ and } \boldu_1,\ldots,\boldu_m \in S, m\in\N
  \right\} \, . \] Then $\Aff_0(S)$ is a vector subspace of
$\R^d$---indeed $\Aff_0(S)$ is the unique translation of the
affine hull of $S$ that contains the origin.

\begin{proposition}\label{prop:rational-supporting-hyperplane}
  The vector space $\Aff_0(\MAX_{\boldtau}(A^\infty(U)))$ has a basis of
  rational vectors and
  $\Aff_0(\MAX_{\boldtau}(A^i(U))) \subseteq \Aff_0(\MAX_{\boldtau}(A^\infty(U)))$
  for any $i\in\N$.
\end{proposition}
\begin{proof}
Let $n\in \N$ be such that 
\begin{gather} 
\Aff_0(\MAX_{\boldtau}(A^i(U))) \subseteq 
\Aff_0(\MAX_{\boldtau}(A^0(U)))+
%  \Aff_0(\MAX_{\boldtau}(A^1(U)))+ 
\cdots + \Aff_0(\MAX_{\boldtau}(A^n(U))) 
\label{eq:aff-short}
\end{gather}
for all $i\in\N$.  Such an $n$ exists since the right-hand
side of (\ref{eq:aff-short}) forms an increasing family of subspaces
of $\R^d$ as $n\rightarrow\infty$ and such a sequence must
eventually stabilize.  We claim that
\begin{gather}
 \Aff_0(\MAX_{\boldtau}(A^\infty(U))) = \Aff_0(\MAX_{\boldtau}(A^0(U)))
% +  \Aff_0(\MAX_{\boldtau}(A^1(U)))
+ \cdots + \Aff_0(\MAX_{\boldtau}(A^n(U))) \, .
\label{eq:aff}
\end{gather}
From the claim it follows that $\Aff_0(\MAX_{\boldtau}(A^\infty(U)))$ has a
basis of rational vectors since $\MAX_{\boldtau}(A^i(U))$ is bounded polytope
with rational vertices for $i=0,\ldots,n$.

It remains to prove the claimed equality (\ref{eq:aff}).  The
left-to-right inclusion follows directly from
Equations (\ref{eq:local-max}) and (\ref{eq:aff-short}).  For the
right-to-left inclusion, it suffices to note that for all $i\in\N$ we have that
$\Aff_0(\MAX_{\boldtau}(A^i(U))) \subseteq \Aff_0(\MAX_{\boldtau}(A^\infty(U)))$.
Indeed, suppose that
$A^i\boldu,A^i\boldv \in \MAX_{\boldtau}(A^i(U))$ for some $\boldu,\boldv \in U$ and
$i\in\N$.  Define $\boldu_j,\boldv_j\in U$ for $j\in\N$ by
$\boldu_i=\boldu$, $\boldv_i=\boldv$, and $A^j\boldu_j=A^j\boldv_j\in\MAX_{\boldtau}(A^j(U))$ for
$j\neq i$.  Then
\[ A^i\boldu-A^i\boldv = \sum_{j=0}^\infty A^j\boldu_j - \sum_{j=0}^\infty A^j\boldv_j \in
  \Aff_0(\MAX_{\boldtau}(A^\infty(U))) \, . \] 

\end{proof}

\begin{proposition}\label{prop:bilinear}
  Let the eigenvalues of matrix $A \in \Q^{d\times d}$ be
  $0<\lambda_1< \ldots < \lambda_k$.  Then there is a 
collection of bilinear forms
$L_{ij}:\R^d\times \R^d \rightarrow \R$ with
algebraic coefficients such that
for all
  $\boldu,\boldtau \in \R^d$ we have
\begin{gather} 
 \Angle{A^n\boldu,\boldtau} = \sum_{i=1}^k \sum_{j=0}^{d-1} \binom{n}{j}\lambda_i^n
    L_{ij}(\boldu,\boldtau) \, .
\label{eq:bilinear}
\end{gather}
\end{proposition}
\begin{proof}
  By the Jordan–Chevalley decomposition, we have $A=P^{-1}DP+N$ where
  $D$ is diagonal, $N$ is nilpotent, $P^{-1}DP$ and $N$ commute, and
  all matrices have algebraic coefficients and can be computed\footnote{The
  eigenvalues are algebraic, being roots of the characteristic polynomial.
  The (generalized) eigenvectors are then algebraic, being solutions to linear
  equations with algebraic coefficients. The matrices $D$ and $P$ can then
  be defined in terms of eigenvalues and (generalized) eigenvectors. See \cite{Cai00} for
  more details.}. Moreover we can write
  $D=\lambda_1D_1 + \cdots + \lambda_kD_k$ for appropriate idempotent
  diagonal matrices $D_1,\ldots,D_k$.  Then for all $n\in \mathbb{N}$
  we have
\begin{eqnarray}
  \Angle{A^n\boldu,\boldtau}
  &=& \boldtau^\top (P^{-1}DP+N)^n\boldu\notag \\
  &=&\boldtau^\top \sum_{j=0}^n \binom{n}{j}P^{-1}D^{n-j}P N^j\boldu \notag\\
  &=& \boldtau^\top \sum_{j=0}^n \binom{n}{j}P^{-1}(\lambda_1^{n-j}D_1+
    \cdots + \lambda_k^{n-j}D_k)PN^j\boldu \notag \\
  &=& \sum_{i=1}^k \lambda_i^n 
   \sum_{j=0}^{d-1} \binom{n}{j}
\underbrace{\lambda_i^{-j}\boldtau^\top P^{-1}D_iPN^j \boldu}_{=:L_{ij}(\boldu,\boldtau)} 
\label{eq:collect} \\
  &=&  \sum_{i=1}^k \sum_{j=0}^{d-1}
       \binom{n}{j} \lambda_i^n L_{ij}(\boldu,\boldtau)
\, , \notag
\end{eqnarray}
where $L_{ij}(\boldu,\boldtau)$  is defined
in (\ref{eq:collect}).  This concludes the proof
because each $L_{i,j}$ is clearly bilinear with algebraic
coefficients.
\end{proof}

Given $\boldtau\in\R^d$, define
$\DOM_{\mathcal{L}}(\boldtau)$ to be the
linear subspace of $\R^d$
comprising all $\boldtau'$
such that for all $\boldu,\boldu'\in\Ext(U)$ and every
bilinear form $L_{ij}$ as in (\ref{eq:bilinear}) if
$L_{ij}(\boldu-\boldu',\boldtau)=0$ then
$L_{ij}(\boldu-\boldu',\boldtau')=0$.
It is clear that $\DOM_{\mathcal{L}}(\boldtau)$
that has a basis of vectors all of whose entries are
algebraic numbers. Indeed, $\boldtau'\in\DOM_{\mathcal{L}}(\boldtau)$
if it satisfies some equations of the form $L_{ij}(\boldu-\boldu',\boldtau')=0$
but $L_{ij}$ is bilinear, has algebraic coefficients and
$\boldu,\boldu'$ have algebraic coefficients.
See Figure~\ref{fig:dominating_set} for a geometrical intuition
of this notion.

\begin{figure*}
    \begin{center}
    \begin{tikzpicture}[
        fillcolor/.style={
            draw=none, fill=darkgreen, fill opacity=0.1
        },
        legend style/.style={
            draw=none, fill=white, inner sep=1.5pt
        },
        ptcolor/.style={draw=blue,circle,inner sep=\fracszpt,fill=blue},
        edgecolor/.style={blue,line width=\fracszedge}, normcolor/.style=red,
        qpoly/.style={draw=red,line width=\fracszedge,pattern color=red, pattern=crosshatch dots},
        qsep/.style={line width=0.5pt,gray},
        qtau/.style={line width=1pt,gray,->},
        conedge/.style={black,dashed}, coneinfedge/.style={black,dashed},
        conefill/.style={fill=gray, fill opacity=0.15},
        scale=1.5]
        \newcommand{\fracszpt}{0.7pt}
        \newcommand{\fracszedge}{0.7pt}
        \coordinate (zero) at (0,0) {};
        \coordinate (vc0) at (0, 1) {};
        \coordinate (vc1) at (2, 0) {};
        \coordinate (vc2) at (1.5, -1) {};
        \coordinate (vc3) at (-1, -1) {};
        \coordinate (vc4) at (-1.3, -9/10) {};
        \coordinate (vc5) at (-1.7, -2/3) {};
        \coordinate (vc6) at (-1.9, -1/3) {};
        \coordinate (vcleft) at (-2, 0) {};
        \coordinate (vcm1) at (-1, 1) {};
        \coordinate (vcm2) at (-1.3, 9/10) {};
        \coordinate (vcm3) at (-1.7, 2/3) {};
        \coordinate (vcm4) at (-1.9, 1/3) {};
        \draw[edgecolor,dotted,thick] (vcleft) -- (vcm4);
        \draw[edgecolor] (vcm4) -- (vcm3);
        \draw[edgecolor,dotted,thick] (vcm3) -- (vcm2);
        \draw[edgecolor] (vcm2) -- (vcm1) -- (vc0) -- (vc1) -- (vc2) -- (vc3) -- (vc4);
        \draw[edgecolor,dotted,thick] (vc4) -- (vc5);
        \draw[edgecolor] (vc5) -- (vc6);
        \draw[edgecolor,dotted,thick] (vc6) -- (vcleft);
        \node[ptcolor] at (vcm4) {};
        \node[ptcolor] at (vcm3) {};
        \node[ptcolor] at (vcm2) {};
        \node[ptcolor] at (vcm1) {};
        \node[ptcolor] at (vc0) {};
        \node[ptcolor] at (vc1) {};
        \node[ptcolor] at (vc2) {};
        \node[ptcolor] at (vc3) {};
        \node[ptcolor] at (vc4) {};
        \node[ptcolor] at (vc5) {};
        \node[ptcolor] at (vc6) {};
        \node[ptcolor] at (vcleft) {};
        \fill[fillcolor]  (vcm4) -- (vcm3) -- (vcm2) -- (vcm1) -- (vc0) -- (vc1) -- (vc2) -- (vc3) -- (vc4) -- (vc5) -- (vc6) -- (vcleft) --cycle;
        \draw (0,0) node[legend style] {$A^*(U)$};

        \newcommand{\tausz}{0.5cm}
        \coordinate (tau) at (1,0.2);
        \draw[qtau] (vc1) -- ++($(zero)!\tausz!(tau)$) node[right] {$\boldtau$};
        \draw[qsep] ($(vc1)!1!90:($(vc1)+(tau)$)$) -- ($(vc1)!1!-90:($(vc1)+(tau)$)$);
        \newcommand{\taueps}{0.5}
        \coordinate (taup) at (0.2,-1);
        \coordinate (tausum) at ($(tau)+\taueps*(taup)$);
        \draw[qtau] (vc1) -- ++($(zero)!\tausz!(tausum)$) node[right] {$\boldtau+\epsilon\boldtau'$};
        \draw[qsep] ($(vc1)!1!90:($(vc1)+(tausum)$)$) -- ($(vc1)!1!-90:($(vc1)+(tausum)$)$);
        \renewcommand{\taueps}{0.5}
        \coordinate (taup) at (-0.9,1.2);
        \coordinate (tausum) at ($(tau)+\taueps*(taup)$);
        \draw[qtau] (vc1) -- ++($(zero)!\tausz!(tausum)$) node[right] {$\boldtau+\epsilon\boldtau'$};
        \draw[qsep] ($(vc1)!1!90:($(vc1)+(tausum)$)$) -- ($(vc1)!1!-90:($(vc1)+(tausum)$)$);

        \draw (vc1) node[draw,circle,very thick,red,inner sep=3pt] (vc1circ) {};
        \draw (3,-1) node[red,anchor=west] (max1text) {$\begin{array}{l}
            \MAX_{\boldtau+\epsilon\boldtau'}(A^\infty(U))\\
            =\MAX_{\boldtau}(A^\infty(U))\end{array}$};
        \draw (max1text.west) edge[red,->,thick,bend left=10,shorten >=1pt] (vc1circ);

        \coordinate (tau) at (-1,0);
        \draw[qtau] (vcleft) -- ++($(zero)!\tausz!(tau)$) node[left] {$\boldtau$};
        \draw[qsep] ($(vcleft)!1!90:($(vcleft)+(tau)$)$) -- ($(vcleft)!1!-90:($(vcleft)+(tau)$)$);
        \coordinate (tau) at ($(vcm4)!(zero)!(vcm3)$);
        \draw[qtau] ($(vcm4)!0.5!(vcm3)$) -- ++($(zero)!\tausz!(tau)$) node[left] {$\boldtau+\epsilon\boldtau'$};
        \draw[qsep] ($(vcm4)!1!90:($(vcm4)+(tau)$)$) -- ($(vcm4)!1!-90:($(vcm4)+(tau)$)$);
        \draw (vcleft) node[draw,circle,very thick,red,inner sep=3pt] (vcleftcirc) {};
        %\draw node[fit=(vcm3) (vcm4),draw,rectangle,red,thick] (vcm34circ) {};
        % trick: we want to rotate the rectangle in such a way it aligned on the edge, the only
        % easy way I know to do this is to redraw the path with no action and use sloped, also we
        % we need to use midway or pos otherwise the tangent to the curve is not right
        % also we need transform shape because we can't set minimum width in independent unit, we
        % need to untransform it so that it cancels with scaling
        \draw let \p1=($(vcm4)-(vcm3)$) in (vcm3) -- (vcm4) node[midway,sloped,rounded corners=\fracszpt*4,
            rectangle,draw,red,very thick,inner sep=\fracszpt*3,minimum width={\fracszpt*6+veclen(\x1,\y1)},transform shape]
            (vcm34circ) {};
        \draw (-2.6,-1.2) node[red,anchor=west,inner sep=0pt] (max2text)
            {$\MAX_{\boldtau}(A^\infty(U))$};
        \draw ($(max2text.north west)!0.2!(max2text.north east)$) edge[red,->,thick,bend left=10,shorten >=1pt] (vcleftcirc);
        \draw (-1.5,1.5) node[red,anchor=east,inner sep=0pt] (max3text)
            {$\MAX_{\boldtau+\epsilon\boldtau'}(A^\infty(U))$};
        \draw ($(max3text.south west)!0.8!(max3text.south east)$) edge[red,->,thick,bend left=10,shorten >=1pt] (vcm34circ);
    \end{tikzpicture}
    \end{center}
    \caption{The set $\DOM_{\mathcal{L}}(\boldtau)$ contains
    directions $\tau'$ such that the sequence $\Angle{A^n(u-u'),\tau'}$ does not asymptotically
    dominate $\Angle{A^n(u-u'),\tau}$ in absolute value (for all $u,u'\in U$).  Intuitively,
    eventual $\tau$ maximizers in $U$ are also eventual $\tau+\varepsilon\tau'$ maximizers
for $\varepsilon$ small enough.  Thus the face of
    $A^\infty(U)$ in direction $\tau+\epsilon\tau'$ is the same as the one in direction $\tau$
    for small $\epsilon$ (see Proposition~\ref{prop:perturb}), as illustrated on the right.
    When $\tau'$ asymptotically dominates $\tau$, any perturbation of $\tau$ will result in a
    different face, as illustrated on the left. In particular this can happen when there is an
    infinite number of faces converging to a single point (as in Figure~\ref{fig:convex-lower-dim-face}).
    \label{fig:dominating_set}}
\end{figure*}
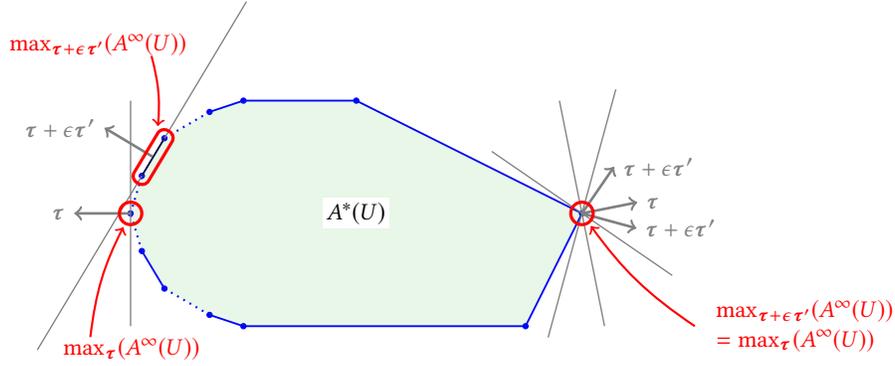

We say that a sequence of real numbers
$\langle x_n : n\in \N \rangle$ is \emph{positive} if $x_n>0$
for all $n\in\N$.  We moreover say that 
$\langle x_n : n\in \N \rangle$ is \emph{ultimately positive} if
there exists $N\in\N$ such that $x_n>0$ for all $n\geq N$.

\begin{proposition}\label{prop:ultpos}
  Suppose that $\boldtau,\boldtau'\in\R^d$
  are such that
  $\boldtau'\in\DOM_{\mathcal{L}}(\boldtau)$.
  Then for all $\boldu,\boldu'\in\Ext(U)$, if the
  sequence
  $\Angle{A^n(\boldu-\boldu'),\boldtau}$ is
  positive (resp. ultimately positive) then there exists $\varepsilon>0$ such that the sequence
  $\Angle{A^n(\boldu-\boldu'),\boldtau+\varepsilon\boldtau'}$
  is also positive (resp. ultimately positive).
\end{proposition}
\begin{proof}
By Proposition~\ref{prop:bilinear} we have that 
\begin{gather}
 \Angle{A^n(\boldu-\boldu'),\boldtau}
 = \sum_{i=1}^k \sum_{j=0}^{d-1} \binom{n}{j} \lambda_i
    L_{ij}(\boldu-\boldu',\boldtau) \, .
\label{eq:ultpos}
\end{gather}
Define a dominance ordering $\prec$ on the terms in (\ref{eq:ultpos}),
where $\binom{n}{j} \lambda_i^n \prec \binom{n}{\ell}\lambda_p^n$ if
$\lambda_i<\lambda_p$ or $\lambda_i=\lambda_j$ and $j<\ell$.  Clearly
$\binom{n}{j}\lambda_i^n \prec \binom{n}{\ell}\lambda_p^n$ implies
that
$\binom{n}{j}\lambda_i^n = o\left(\binom{n}{\ell}\lambda_p^n\right)$
as $n\rightarrow\infty$.

Let $\binom{n}{j_0}\lambda_{i_0}^n$ be the dominant term in the expansion
(\ref{eq:ultpos}) of
$\Angle{A^n(\boldu-\boldu'),\boldtau}$. Since
$\boldtau' \in \DOM_{\mathcal{L}}(\boldtau)$, the
expansion of
$\Angle{A^n(\boldu-\boldu'),\boldtau'}$
does not have any term that strictly dominates $\binom{n}{j_0}\lambda^{i_0}$.
It follows that $\Angle{A^n(\boldu-\boldu'),\boldtau'}
=O(\Angle{A^n(\boldu-\boldu'),\boldtau})$ as $n\rightarrow\infty$.
In particular, there exists absolute constants $N\in\N$ and $K>0$ such that
\[
|\Angle{A^n(\boldu-\boldu'),\boldtau})|
    \leqslant K\Angle{A^n(\boldu-\boldu'),\boldtau}
\]
for all $n\geqslant N$, since the sequence
$\Angle{A^n(\boldu-\boldu'),\boldtau}$ is positive (resp. ultimately
positive, in which case we just have to take $N$ large enough).
Then for any $\epsilon>0$ and $n\geqslant N$,
\begin{align*}
\Angle{A^n(\boldu-\boldu'),\boldtau+\varepsilon\boldtau'}
  &= \Angle{A^n(\boldu-\boldu'),\boldtau}+
  \varepsilon\Angle{A^n(\boldu-\boldu'),\boldtau'}\\
  &\geqslant (1-\epsilon K)\Angle{A^n(\boldu-\boldu'),\boldtau}>0
\end{align*}
if $\epsilon<K^{-1}$. In particular, this proves that
$\Angle{A^n(\boldu-\boldu'),\boldtau+\varepsilon\boldtau'}$
is ultimately positive. On the other hand, for $n<N$, observe that since
\[
  \Angle{A^n(\boldu-\boldu'),\boldtau+\varepsilon\boldtau'}
  = \Angle{A^n(\boldu-\boldu'),\boldtau}+
  \varepsilon
  \Angle{A^n(\boldu-\boldu'),\boldtau'}\, ,\]
by making $\varepsilon$ suitably small we can ensure that
$\Angle{A^n(\boldu-\boldu'),\boldtau+\varepsilon\boldtau'}$
is positive if $\Angle{A^n(\boldu-\boldu'),\boldtau}$
is positive.
\end{proof}

\begin{proposition}\label{prop:perturb}
  Let $\boldtau,\boldtau'\in \R^d$ be such
  that
  $\boldtau'\in\DOM_{\mathcal{L}}(\boldtau)$
  and
  $\boldtau'\in\Aff_0(\MAX_{\boldtau}(A^\infty(U)))^\perp$.
  Then for $\varepsilon$ sufficiently small we have
  $\MAX_{\boldtau}(A^\infty(U)) \subseteq
  \MAX_{\boldtau+\varepsilon\boldtau'}(A^\infty(U))$.
\end{proposition}
\begin{proof}
Define the set of \emph{eventually $\boldtau$-maximizing inputs}
to be
\[\begin{array}{r@{\,}l}
U_{\boldtau}^{\MAX} = \{ \boldu \in \Ext(U) : \forall \boldv\in \Ext(U):
  &\Angle{A^n(\boldu-\boldv),\boldtau} \text{ is identically}\\
  &\text{0 or ultimately positive}\}\,.
  \end{array}\]
We claim that $U_{\boldtau}^{\MAX}$ is non-empty.  To see this, note that
every sequence of the form $\Angle{A^n(\boldu-\boldv),\boldtau}$ is either
identically zero or ultimately negative or ultimately positive
(depending on the sign of its leading term).  Thus we can define a linear 
preorder on the vertices of $U$ in which $\boldu$ 
greater than $\boldv$ if 
the sequence $\Angle{A^n(\boldu-\boldv),\boldtau}$ is ultimately
non-negative. Since $U$ has a finite number of vertices, there is a
maximal element under this preorder, which is then an element of
$U_{\boldtau}^{\MAX}$.  This establishes the claim.  We
further observe,
that if
$\sum_{i=0}^\infty A^i\boldu_i \in \MAX_{\boldtau}(A^\infty(U))$ then there
exists $N\in\N$ such that for all $i\geq N$,
$\boldu_i\in\mathrm{Conv}(U_{\boldtau}^{\MAX})$, i.e., after some point 
the input $\boldu_i$ must be chosen in the convex hull of the set
of eventually $\boldtau$-maximizing inputs.

We claim that for $\varepsilon>0$ small enough we have $\MAX_{\boldtau}(A^i(U))
\subseteq \MAX_{\boldtau+\epsilon\boldtau'}(A^i(U))$ for $i=0,\ldots,N-1$.
Indeed, $\Aff_0(\MAX_{\boldtau}(A^i(U)))
\subseteq \Aff_0(\MAX_{\boldtau}(A^\infty(U)))$
by Proposition~\ref{prop:rational-supporting-hyperplane}, thus
$\boldtau'\in\Aff_0(\MAX_{\boldtau}(A^i(U)))^\perp$.
Let $A^i\boldu\in\MAX_{\boldtau}(A^i(U))$ and $\boldv\in U$,
then $\Angle{A^i(\boldu-\boldv),\boldtau}\geqslant0$. There are two cases to consider:
if $\Angle{A^i(\boldu-\boldv),\boldtau}=0$ then
$A^i\boldv\in\MAX_{\boldtau}(A^i(U))$ thus
$\Angle{A^i(\boldu-\boldv),\boldtau+\epsilon\boldtau'}=
\epsilon\Angle{A^i(\boldu-\boldv),\boldtau'}=0$.
If $\Angle{A^i(\boldu-\boldv),\boldtau}>0$ then
$\Angle{A^i(\boldu-\boldv),\boldtau+\epsilon\boldtau'}=
\Angle{A^i(\boldu-\boldv),\boldtau+\epsilon\boldtau'}\geqslant 0$
for small enough $\epsilon$.

We claim that
$U_{\boldtau}^{\MAX} \subseteq
  U_{\boldtau+\varepsilon\boldtau'}^{\MAX}$ 
for $\varepsilon$ sufficiently small.  Aside we have seen that for $\varepsilon$ small enough we have
  $\MAX_{\boldtau}(A^i(U)) \subseteq \MAX_{\boldtau+\epsilon\boldtau'}(A^i(U))$ for
  $i=0,\ldots,N-1$.  It follows that for $\varepsilon$ sufficiently
  small we have
  $\MAX_{\boldtau}(A^\infty(U)) \subseteq
  \MAX_{\boldtau+\varepsilon\boldtau'}(A^\infty(U))$, as we wanted to prove.

  It remains to prove the claim.  To this end, consider
  $\boldu \in U_{\boldtau}^{\MAX}$ and
  $\boldv\in \Ext(U)$.  If the sequence
  $\Angle{A^n(\boldu-\boldv),\boldtau}$ is
  identically zero then also
  $\boldv \in U_{\boldtau}^{\MAX}$.  Since
  $\boldtau' \in \Aff_0(\MAX_{\boldtau}(A^i(U)))^\perp$
  for all $i\in\N$ it follows that
  $\Angle{A^n(\boldu-\boldv),\boldtau+\varepsilon\boldtau'}$
  is also identically zero.  On the other hand, if
  $\Angle{A^n(\boldu-\boldv),\boldtau}$ is
  ultimately positive then for $\epsilon$ small enough
  $\Angle{A^n(\boldu-\boldv),\boldtau+\varepsilon\boldtau'}$
  is also ultimately positive by Proposition~\ref{prop:ultpos}.  This
  proves the claim and concludes the proof.
\end{proof}

\begin{proposition}\label{prop:one-dim}
If $\boldtau$ is an extremal vector in 
$\CONE(A^\infty(U),Q)$ then 
\begin{equation} \DOM_{\mathcal{L}}(\boldtau) \cap \Aff_0(\MAX_{\boldtau}(A^\infty(U)) \cup
\MIN_{\boldtau}(Q))^\perp = \Span(\boldtau) \, . 
\label{eq:extremal}
\end{equation}
\end{proposition}
\begin{proof}
  We first show the right-to-left inclusion, for which it suffices to
  show that $\boldtau$ lies in the left-hand side.  It is clear that
  $\boldtau \in \Aff_0(\MAX_{\boldtau}(A^\infty(U)))^\perp$ and
  $\boldtau \in \Aff_0(\MIN_{\boldtau}(Q))^\perp$ (e.g., if
  $\boldu,\boldv \in \MIN_{\boldtau}(Q)$ then we have
  $\Angle{\boldu,\boldtau} = \Angle{\boldv,\boldtau}$ and hence $\Angle{\boldu-\boldv,\boldtau}=0$).
  Now consider $\boldu \in \MAX_{\boldtau}(A^\infty(U))$ and $\boldv\in \MIN_{\boldtau}(Q)$.
  By definition we have $\Angle{\boldu,\boldtau}\leq \Angle{\boldv,\boldtau}$.  But if
  $\Angle{\boldu,\boldtau} < \Angle{\boldv,\boldtau}$, then since both $A^\infty(U)$ and
  $Q$ are bounded we have that $\boldtau$ lies in the interior of
  $\CONE(A^\infty(U),Q)$ contradicting the assumption that
  $\boldtau$ is extremal.  Thus we must have
  $\Angle{\boldu,\boldtau}= \Angle{\boldv,\boldtau}$ and hence $\Angle{\boldu-\boldv,\boldtau}=0$.  We
  conclude that
  $\boldtau \in \Aff_0(\MAX_{\boldtau}(A^\infty(U)) \cup \MIN_{\boldtau}(Q))^\perp$.
  This completes the proof of the right-to-left inclusion.

  For the left-to-right inclusion, consider a vector $\boldtau'$ contained
  in the left-hand side of (\ref{eq:extremal}).  We claim that for
  suitably small $\varepsilon\in\R$, both
  $\boldtau+\varepsilon\boldtau'$ and $\boldtau-\varepsilon\boldtau'$ lie in
  $\CONE(A^\infty(U),Q)$.  Since $\boldtau$ is extremal in
  $\CONE(A^\infty(U),Q)$ we conclude that
  $\boldtau'\in\Span(\boldtau)$.

  It remains to prove the claim.  To this end, notice that for
  $\boldu\in\MAX_{\boldtau}(A^\infty(U))$,
  $\boldv\in \MIN_{\boldtau}(Q)$, and
  $\varepsilon\in\R$, we have
  $\Angle{\boldu-\boldv,\boldtau+\varepsilon\boldtau'}
  = \Angle{\boldu-\boldv,\boldtau} = 0$.
  Moreover by Proposition~\ref{prop:perturb}, for $\varepsilon$
  suitably small we have
  $\MAX_{\boldtau}(A^\infty(U)) \subseteq
  \MAX_{\boldtau+\varepsilon\boldtau'}(A^\infty(U))$.
  Similar but simpler reasoning to the proof of
  Proposition~\ref{prop:perturb} also yields that
  $\MIN_{\boldtau}(Q) \subseteq
  \MIN_{\boldtau+\varepsilon\boldtau'}(Q)$ for
  $\varepsilon$ small enough.  It follows that
  $\Angle{\boldu,\boldtau+\varepsilon\boldtau'}
  = \Angle{\boldv,\boldtau+\varepsilon\boldtau'}$ for all
  $\boldu\in
  \MAX_{\boldtau+\varepsilon\boldtau'}(A^\infty(U))$
  and all
  $\boldv\in
  \MIN_{\boldtau+\varepsilon\boldtau'}(Q)$.  Thus
  $\boldtau+\varepsilon\boldtau'$ separates $A^*(U)$
  from $Q$, establishing the claim.
\end{proof}

We can now show our separation lemma.
\begin{proof}[Proof of Lemma~\ref{lem:algebraic_separator}]
Recall that $\CONE(A^\infty(U),Q)$ is a closed pointed cone.
It follows that it has an extremal vector $\boldtau$.  By
Proposition~\ref{prop:rational-supporting-hyperplane}
and Proposition~\ref{prop:one-dim} we can assume that all entries of
$\boldtau$ are algebraic. Indeed we have already noted that
$\DOM_{\mathcal{L}}(\boldtau)$ and $\Aff_0(\MIN_{\boldtau}Q))$
have a basis of rational vectors. We conclude that if there is a
separator of $A^*(U)$ and $Q$ then there is an algebraic separator.
\end{proof}

\begin{theorem}
The reachability problem for simple LTI systems is decidable.
\end{theorem}
\begin{proof}
  As we have noted above, it suffices to give a semi-decision
  procedure to show that a target $Q\subseteq\R^d$ is not
  reachable in a given LTI system $\mathcal{L}=(A,U)$.  To show this
  we enumerate all vectors $\boldtau\in\R^d$ with
  algebraic entries and determine whether $\boldtau$
  separates $A^\infty(U)$ from $Q$, that is, whether
\begin{gather} \{ \Angle{\boldu,\boldtau} : \boldu\in A^\infty(U) \}
\leq \min\{ \Angle{\boldv,\boldtau} : \boldv \in Q \} \, .
\label{eq:separate}
\end{gather}

It is straightforward to calculate right-hand side in
(\ref{eq:separate}).  To compute left-hand side the idea is to find
some eventually $\tau$-maximising input $\boldu \in \Ext(U)$
and corresponding threshold $N\in\N$ such that for all
$\boldv\in \Ext(U)$ and all $i\geq N$ we have
$\Angle{A^i \boldu,\boldtau} \geq \Angle{A^i
  \boldv,\boldtau}$.  Given $\boldu$ and $N$,
we obtain an element of $\MAX_{\boldtau}(A^\infty(U))$ as
\[ \sum_{i=0}^{N-1}A^i\boldu_i + A^N\sum_{i=0}^\infty A^i \boldu =
   \sum_{i=0}^{N-1}A^i\boldu_i + A^N(I-A)^{-1}\boldu \]
where $\boldu_i \in \MAX_{\boldtau}(A^i(U))$ for $i=0,\ldots,N-1$.

It remains to find such a $\boldu$ and $N$.  For this we
consider each sequence of the form
$\Angle{A^i(\boldv-\boldw),\boldtau}$ for
$\boldv,\boldw\in \Ext(U)$.  Each such sequence is
either identically zero, ultimately positive, or ultimately negative.
Moreover by examining the dominant term of the sequence we can decide
which of these eventualities is the case and, in case of an ultimately
positive sequence, compute the index from which the sequence becomes
positive.  Clearly this is enough to determine which of the extremal
points of $U$ is an ultimate $\boldtau$-maximiser and to
determine the corresponding threshold $N$.
\end{proof}

\section{Conclusion}
Our main result showed decidability of the LTI Reachability Problem
for so-called simple LTI systems.  The most restrictive condition in
the notion of a simple LTI system (see Section~\ref{sec:simple}) is
that some power of the transition matrix have real spectrum.  This
assumption was crucial in proving that every unreachable point is
separated from the set of reachable points by a hyperplane with
algebraic coefficients.  To illustrate the difficulty with eliminating
or weakening this assumption, consider the following LTI system in 
which the transition matrix performs a counter clockwise
rotation in the plane by angle $\theta$:
\[A=\frac{1}{2}\begin{pmatrix}\cos\theta&-\sin\theta\\\sin\theta&\cos\theta\end{pmatrix},
  \qquad U=[0,1]\times\set{0}.\] Considering the direction
$\tau=(1,0)$, it is not hard to see that the furthest we can go in
direction $\tau$ is
\[\max_{x\in
    A^\infty(U)}\Angle{x,\tau}=
\sum_{n=0}^\infty\max(0,2^{-n}\cos(n\theta))  \, .
\]
Now $A$ has eigenvalues $\lambda=\tfrac{1}{2}e^{i\theta}$ and
$\overline{\lambda}=\tfrac{1}{2}e^{-i\theta}$. If $A$ is simple then a
positive power of $\lambda$ must be real, in other words $\theta$ is a
rational multiple of $\pi$.  In this case the sign of $\cos(n\theta)$
is periodic and we can find an explicit expression for the sum above
and thereby show that it is an algebraic number.  However if $\theta$
is an irrational multiple of $\pi$ and the sign of $\cos(n\theta)$ is
hard to analyse. In particular, we do not know if the resulting sum is
an algebraic number.

\bibliographystyle{plainurl}
\bibliography{refs}

\newpage
\appendix

\section{Undecidability for Invertible Matrix Problems}

Given $k+1$ invertible matrices $A_{1}, \ldots, A_{k},C \in \Q^{d \times d}$, the
\emph{generalized matrix powering problem for invertible matrices} consists in deciding whether
there exist $n_{1}, \ldots, n_{k} \in \Z \setminus \set{0}$ such that
\begin{equation*}
\prod\limits_{i=1}^{k} A_{i}^{n_{i}} = C.
\end{equation*}

The following results are folklore but we could not find any proof of them in the literature.

\begin{theorem}
The generalized matrix powering problem for invertible matrices is undecidable.
\end{theorem}

\begin{proof}
We will show this result by reducing from Hilbert's Tenth Problem. Given a polynomial $p \in \Z[n_{1}, \ldots, n_{k}]$, it is easy to express $p(n_{1}, \ldots, n_{k})$ as a conjunction of relations of the following form (noting that we may need to introduce new variables):
\begin{itemize}
\item $z = k$, where $k \in \Z$
\item $z = x+y$
\item $z = xy$.
\end{itemize}
We start by showing how to encode each of these as an instance of the generalized matrix powering problem for invertible matrices.
Firstly, note that
\begin{equation*}
z = k \Leftrightarrow
\begin{pmatrix}
    1 & 1 \\
    0 & 1
\end{pmatrix}^{z} =
\begin{pmatrix}
    1 & k \\
    0 & 1
\end{pmatrix}.
\end{equation*}
Secondly, note that
\begin{equation*}
    z = x + y \Leftrightarrow
\begin{pmatrix}
    1 & 1 \\
    0 & 1
\end{pmatrix}^{x}
\begin{pmatrix}
    1 & 1 \\
    0 & 1
\end{pmatrix}^{y}
\begin{pmatrix}
    1 & -1 \\
    0 & 1
\end{pmatrix}^{z} =
\begin{pmatrix}
    1 & 0 \\
    0 & 1
\end{pmatrix}.
\end{equation*}
Thirdly, note that
\begin{equation*}
    z = xy \Leftrightarrow
    \exists x',y' \in \Z,
    \begin{pmatrix}
        1 & 0 & 0 \\
        0 & 1 & 0 \\
        0 & 0 & 1
    \end{pmatrix} =
    \begin{pmatrix}
        1 & x-x' & z-xy \\
        0 & 1 & y-y' \\
        0 & 0 & 1
    \end{pmatrix}
\end{equation*}
and that the latter matrix is just equal to
\begin{equation*}
    \begin{pmatrix}
        1 & 0 & -1 \\
        0 & 1 & 0 \\
        0 & 0 & 1
    \end{pmatrix}^{z}
    \begin{pmatrix}
        1 & 0 & 0 \\
        0 & 1 & -1 \\
        0 & 0 & 1
    \end{pmatrix}^{y'}
    \begin{pmatrix}
        1 & 1 & 0 \\
        0 & 1 & 0 \\
        0 & 0 & 1
    \end{pmatrix}^{x}
    \begin{pmatrix}
        1 & 0 & 0 \\
        0 & 1 & 1 \\
        0 & 0 & 1
    \end{pmatrix}^{y}
    \begin{pmatrix}
        1 & -1 & 0 \\
        0 & 1 & 0 \\
        0 & 0 & 1
    \end{pmatrix}^{x'}.
\end{equation*}
Finally, conjunction can be achieved by making use of separate matrix blocks:
\begin{equation*}
    \prod\limits_{i=1}^{k} A_{i}^{n_{i}} = C \wedge \prod\limits_{i=1}^{k} B_{i}^{n_{i}} = D \Leftrightarrow
    \prod\limits_{i=1}^{k} \begin{pmatrix}A_{i} & 0 \\ 0 & B_{i}\end{pmatrix}^{n_{i}} = \begin{pmatrix}C & 0 \\ 0 & D\end{pmatrix}.
\end{equation*}
\end{proof}

\begin{definition}
Given invertible matrices $A_{1}, \ldots, A_{k} \in \Q^{d \times d}$ and two non-zero vectors $\boldsymbol{x}, \boldsymbol{y} \in \Q^{d}$, the \emph{vector reachability problem for invertible matrices} consists in deciding whether there exist $n_{1}, \ldots, n_{k} \in \Z \setminus \set{0}$ such that
\begin{equation*}
\prod\limits_{i=1}^{k}A_{i}^{n_{i}} \boldsymbol{x} = \boldsymbol{y}.
\end{equation*}
\end{definition}

\begin{theorem}
The vector reachability problem for invertible matrices is undecidable.
\end{theorem}

\begin{proof}
This can be shown by reduction from the generalised matrix powering problem for invertible matrices. In particular, given invertible matrices $A_{1}, \ldots, A_{k}, B \in \Q^{d \times d}$, letting $\boldsymbol{b}_{1}, \ldots, \boldsymbol{b}_{d}$ denote the columns of $B$, and letting $\boldsymbol{e}_{1}, \ldots, \boldsymbol{e}_{d}$ denote the canonical basis of $\R^{d}$, the result follows from the fact that
    \begin{equation*}
        \prod\limits_{i=1}^{k} A_{i}^{n_{i}} = B \Leftrightarrow
        \prod\limits_{i=1}^{k}
        \begin{pmatrix}
            A_{i} & \cdots & 0 \\
            \vdots& \ddots & \vdots \\
            0 & \cdots & A_{i}
        \end{pmatrix}^{n_{i}}
        \begin{pmatrix}
            \boldsymbol{e}_{1} \\
            \vdots \\
            \boldsymbol{e}_{d}
        \end{pmatrix} =
        \begin{pmatrix}
            \boldsymbol{b}_{1} \\
            \vdots \\
            \boldsymbol{b}_{d}
        \end{pmatrix}.
    \end{equation*}
\end{proof}

\end{document}